\documentclass{article}

\usepackage{arxiv}

\usepackage[utf8]{inputenc} 
\usepackage[T1]{fontenc}    
\usepackage{hyperref}       
\usepackage{url}            
\usepackage{booktabs}       
\usepackage{amsfonts}       
\usepackage{nicefrac}       
\usepackage{microtype}      
\usepackage{lipsum}
\usepackage{amsmath,amsthm,amssymb,bm,enumerate,mathtools,soul,cancel,graphicx,booktabs,float,hyperref,multirow,xcolor,subcaption,dsfont}
\usepackage[scr=boondoxo,scrscaled=1.05]{mathalfa}
\usepackage[inline]{trackchanges}
\usepackage{mdframed}
\usepackage[linesnumbered,ruled,vlined]{algorithm2e}

\newtheorem{theorem}{Theorem}
\newtheorem{definition}{Definition PE.}

\addeditor{stp}

\title{\emph{RKHS} embedding for estimating nonlinear piezoelectric systems}

\author{
  Sai Tej Paruchuri, Jia Guo, Andrew Kurdila 
    \\
  Department of Mechanical Engineering\\
  Virginia Tech\\
  Blacksburg, VA 24060 \\
  \texttt{saitejp@vt.edu, jguo18@vt.edu, kurdila@vt.edu} \\
}

\begin{document}
\maketitle

\begin{abstract}
Nonlinearities in piezoelectric systems can arise from internal factors such as nonlinear constitutive laws or external factors like realizations of boundary conditions. It can be difficult or even impossible to derive detailed models from the first principles of all the sources of nonlinearity in a system. As a specific example, in traditional modeling techniques that use electric enthalpy density with higher-order terms, it can be problematic to choose which polynomial nonlinearities are essential. This paper introduces adaptive estimator techniques to estimate the nonlinearities that can arise in certain piezoelectric systems. Here an underlying assumption is that the nonlinearities can be modeled as functions in a reproducing kernel Hilbert space (RKHS). Unlike traditional modeling approaches, the approach discussed in this paper allows the development of models without knowledge of the precise form or structure of the nonlinearity. This approach can be viewed as a data-driven method to approximate the unknown nonlinear system. This paper introduces the theory behind the adaptive estimator and studies the effectiveness of this approach numerically for a class of nonlinear piezoelectric composite beams.
\end{abstract}


\section{Introduction}

Researchers have studied piezoelectric systems extensively over the past three decades for applications to a number of classical problems of vibration attenuation, as well as modern problems like energy harvesting. Even though many of these studies model piezoelectric oscillators as linear systems, piezoelectric systems are often inherently nonlinear. At low input amplitudes, the effect of nonlinearity is ordinarily not very pronounced. However, linear models can fail to capture the dynamics of piezoelectric systems that undergo large displacements, velocities, accelerations, or electric field strengths \cite{Kurdila2016,Tiersten1969,Erturk2011}. Researchers have consequently also developed nonlinear models for many examples of piezoelectric oscillators. Conventionally, the derivation of nonlinear models starts with the inclusion of higher-order (polynomial) terms in the electric enthalpy density \cite{Maugin1986, Yang2006, VONWAGNER2002, VonWagner2003, Wagner2003, VonWagner2004, Stanton2010, Stanton2012}. Using the extended Hamilton's principle or Lagrange's equations then gives a corresponding set of nonlinear equations of motion. The choice of which higher-order terms should be included in the electric enthalpy density largely depends on the amount and type of nonlinearity in the system at hand, and these factors are determined from experiments.

Whether the system is linear or nonlinear, system identification methods are used to estimate unknown dynamics. Over the years, a large inventory of methods have been derived for linear methods that estimate or identify the dynamics of systems in vibrations. See \cite{Ewins1984,Banks1996} for a popular account of the estimation techniques in this field. Not too surprisingly, well-conceived and popular identification methods based on assumptions of linearity can yield questionable results if the underlying system has substantial nonlinearities. We focus primarily here on just one example and the reader is referred to \cite{Ljung2001,Schoukens2019} for a more comprehensive account of linear and nonlinear methodologies. It is important to note that often experimental data for vibrating systems is collected or processed in the frequency domain. Such experiments are not always amenable to the representation of system nonlinearity. To see this, consider the harmonic balancing method, which is used traditionally to study the steady-state response of nonlinear systems \cite{Worden2019}. This method can be viewed as an analytical equivalent of stepped sine testing methods that are popular in nonlinear experiments. From first principles the higher-order harmonic terms arising from system nonlinearity are neglected while implementing the harmonic balance method. Hence, nonlinear models constructed from this method based on frequency domain experiments can fail to capture the entire dynamics of the system. Furthermore, nonlinearity in vibrating systems can arise from external factors like improper or uncertain clamping that imposes boundary conditions. The inclusion of higher-order terms in the electric enthalpy density of a piezoelectric system will not in general capture the nonlinearities resulting from such external factors.

In view of these observations regarding difficulties modeling nonlinear systems, over the past few years, there has been a keen interest in deriving representations of systems using data-driven modeling. While data-driven modeling methods can be understood as a type or class of system identification technique, these methods have emerged as a distinctive discipline over the past few years. These techniques rely on developing models from time or frequency domain experimental data. For vibrating systems where linear models are sufficient, robust techniques like vector fitting can be used to get state-space models from frequency domain experimental data \cite{Gustavsen1999, Malladi2018}. However, data-driven modeling approaches specifically developed for nonlinear systems are an area of increasing interest and are yet to be fully developed. One example of such a technique is the Dynamic Mode Decomposition (DMD) method, which approximates Koopman modes to model the inherent dynamics \cite{Kutz2016, Tu2014, Bravo2017, Bravo2017a, Bravo2019}. See \cite{Kurdila2018} for a discussion of rates of convergence for many of these methods. However, dynamic mode decomposition produces a linear (perhaps infinite-dimensional) model for nonlinear systems, and the accuracy of these techniques remains an open and active research topic.

In this paper, we introduce a novel data-driven approach for modeling nonlinear systems, one that we apply to model nonlinear piezoelectric systems. This approach is based on embedding the unknown nonlinear function appearing in the governing equation in a reproducing kernel Hilbert space (RKHS). The unknown function is subsequently estimated through adaptive parameter estimation. Identification methods that use RKHS have been studied for problems like terrain measurement \cite{Bobade2019}, control of dynamical systems \cite{Joshi2018, Axelrod2015, Chowdhary2013}, sensor selection \cite{Maske2018}, and learning spatiotemporally-evolving systems \cite{Whitman2017, Kingravi2015}. In this paper, we extend the methodology initially developed for autonomous systems in \cite{Bobade2019}, and apply it to the nonlinear piezoelectric systems, a nonautonomous system. The advantages are as follows:
\begin{enumerate}
    \item Under some conditions, this technique provides a bound on the error between the actual and estimated unknown function.
    \item There is a geometric interpretation of the error estimate, in terms of the positive limit set of the system equations, that describes the subset over which convergence is guaranteed. This is a newly observed property of the RKHS embedding method.
    \item This technique not only gives us a nonlinear model but also estimates the underlying nonlinear function over a subspace of the state space. Techniques like DMD generate approximations of an observable but do not usually estimate the underlying system nonlinearity.
    \item Since the primary assumption is that the nonlinear function belongs to an RKHS, this technique can be implemented for a large class of nonlinearities.
    \item Unlike conventional modeling techniques, the source of nonlinearity does not influence the estimation approach. The exact form of the nonlinearity is unknown.
\end{enumerate}


In this study, we take as a prototypical example of a piezoelectric system, a piezoelectric composite beam subject to base excitation, and we model its dynamics using an adaptive estimation technique based on the RKHS embedding method.

\section{Nonlinear Piezoelectric Model}
\label{sec_model}
In this section, we derive the equations of motion for the target class of nonlinear piezoelectric composites. This section carefully describes the precise nature of some constitutive nonlinearities and the limitations of the traditional linear models. In the current study, we have chosen the electric enthalpy density for nonlinear continua given in \cite{VONWAGNER2002} to serve as the means to construct the governing equations and formulate the RKHS embedding approach. Note that the RKHS embedding techniques discussed in this paper are not limited to this problem and can be adapted to model other types of similar nonlinear electromechanical composite oscillators.
\subsection{Nonlinear Electric Enthalpy Density}
The expression for electric enthalpy density for modeling linear piezoelectric continua is given by
\begin{align*}
    \mathcal{H} = \frac{1}{2} C_{ijkl}^E S_{ij} S_{kl} - e_{mij} S_{ij} E_m - \frac{1}{2} \epsilon_{im}^S E_i E_m,
\end{align*}
where $C_{ijkl}^E$, $S_{ij}$, $e_{mij}$, $E_m$, and $\epsilon_{im}^S$ are the Young's modulus, strain, piezoelectric coupling, electric field, and permittivity tensors, respectively. The quadratic form above is written using the summation convention. Based on thermodynamic considerations, the stress and electric displacement, $T_{ij}$ and $D_i$, respectively, are defined in the relations
\begin{align*}
    T_{ij} = \left. \frac{\partial \mathcal{H}}{\partial S_{ij}} \right|_{s,E},  \hspace{0.5in}
    -D_i = \left. \frac{\partial \mathcal{H}}{\partial E_i} \right|_{s,S}.
\end{align*}
The associated constitutive laws of linear piezoelectricity have the form
\begin{align*}
    \begin{Bmatrix} T_{ij} \\ D_m \end{Bmatrix}
    =
    \begin{bmatrix} C^E_{ijkl} & -e_{nij} \\ e_{mkl} & \epsilon^S_{mn} \end{bmatrix}
    \begin{Bmatrix}S_{kl} \\ E_n  \end{Bmatrix} ,
\end{align*}
where again the summation convention holds in the expression above. In the above equations, the superscripts on $C^E_{ijkl}$ and $\epsilon^S_{mn}$ emphasize that these constants are measured when the electric field and strain are held constant. For piezoelectric beam bending models, consideration is restricted to constitutive laws that have the form
\begin{align*}
    \begin{Bmatrix} T_{x} \\ D_z \end{Bmatrix}
    =
    \begin{bmatrix} C^E_{xx} & -e_{zx} \\ e_{zx} & \epsilon^S_{zz} \end{bmatrix}
    \begin{Bmatrix}S_{x} \\ E_z  \end{Bmatrix} ,
\end{align*}
where $x \sim 11$, $z \sim 3$ are the coordinate directions depicted in Figure \ref{fig_pbeam}. The coordinate $x$ is measured along the neutral axis that extends along the length of the beams, and $z$ is in the transverse bending displacement direction. The permittivity at constant strain can be related to that at constant stress using the relation
\begin{align*}
    \epsilon^S_{zz} = \epsilon^T_{zz} - d_{zx}^2 C^E_{xx}.
\end{align*}
The piezoelectric strain coefficient $d_{zx}$ is related to the piezoelectric coupling constant $e_{zx}$ by the equation $e_{zx} = C^E_{xx}d_{zx}$. The constitutive laws for the piezoelectric composite are
\begin{align*}
    T_{x} &= C^E_{xx} S_{x} - d_{zx} C^E_{xx} E_z, \\
    D_z &= d_{zx} C^E_{xx} S_{x} + (\epsilon^T_{zz} - d_{zx}^2 C^E_{xx}) E_z.
\end{align*}
A detailed discussion of this linear case can be found in \cite{Kurdila2016,Tiersten1969,Leo2007b}.

For large values of the field variables, the effects of nonlinearity in the piezoelectric continua can become dominant. We account for these effects by adding higher order terms in the expression for the electric enthalpy density. The nonlinear dependence between $C^E_{xx}$, $d_{zx}$ and $S_x$ can be approximated using the relations \cite{VONWAGNER2002}
\begin{align*}
    C^E_{xx} &= C^{E(0)}_{xx} + C^{E(1)}_{xx} S_{x} + C^{E(2)}_{xx} S_{x}^2, \\
    d_{zx} &= d_{zx}^{(0)} + d_{zx}^{(1)}S_{x} + d_{zx}^{(2)} S_{x}^2.
\end{align*}

The corresponding electric enthalpy density has the form
\begin{align}
    \mathcal{H} = \frac{1}{2} C^{E(0)}_{xx} S_x^2 + \frac{1}{3} C^{E(1)}_{xx} S_x^3 + \frac{1}{4} C^{E(2)}_{xx} S_x^4 - \gamma_0 S_x E - \frac{1}{2} \gamma_1 S_x^2 E - \frac{1}{2} \nu_0 E^2
    \label{eq_eed}
\end{align}
with
\begin{align*}
    \nu_0 &= \epsilon^T - (d_{zx}^{(0)})^2 C^{E(0)}_{xx}, & \gamma_0 = C^{E(0)}_{xx} d_{zx}^{(0)}, \\
    \gamma_1 &= C^{E(0)}_{xx} d_{zx}^{(1)} + C^{E(1)}_{xx} d_{zx}^{(0)}, & \gamma_2 = C^{E(0)}_{xx} d_{zx}^{(2)} + C^{E(2)}_{xx} d_{zx}^{(0)} + C^{E(1)}_{xx} d_{zx}^{(1)}.
\end{align*}

Thus, the nonlinear constitutive laws, obtained using the relations shown above, have the form
\begin{align*}
    T_x &= C^{E(0)}_{xx} S_x + C^{E(1)}_{xx} S_x^2 + C^{E(2)}_{xx} S_x^3 - \gamma_0 E_z - \gamma_1 S_x E_z - \gamma_2 S_x^2 E_z, \\
    D_z &= \gamma_0 S_x + \frac{1}{2} \gamma_1 S_x^2 + \frac{1}{3} \gamma_2 S_x^3 + \nu_0 E_z.
\end{align*}

See references \cite{Stanton2010,Stanton2012,Cottone2012,Friswell2012,Rakotondrabe2009} for other similar models that are used to represent the behavior of nonlinear piezoelectric systems.

\subsection{Equations of Motion}
\label{sec_ham_prin}

\begin{figure}[htb!]
\centering
\includegraphics[scale = 0.55]{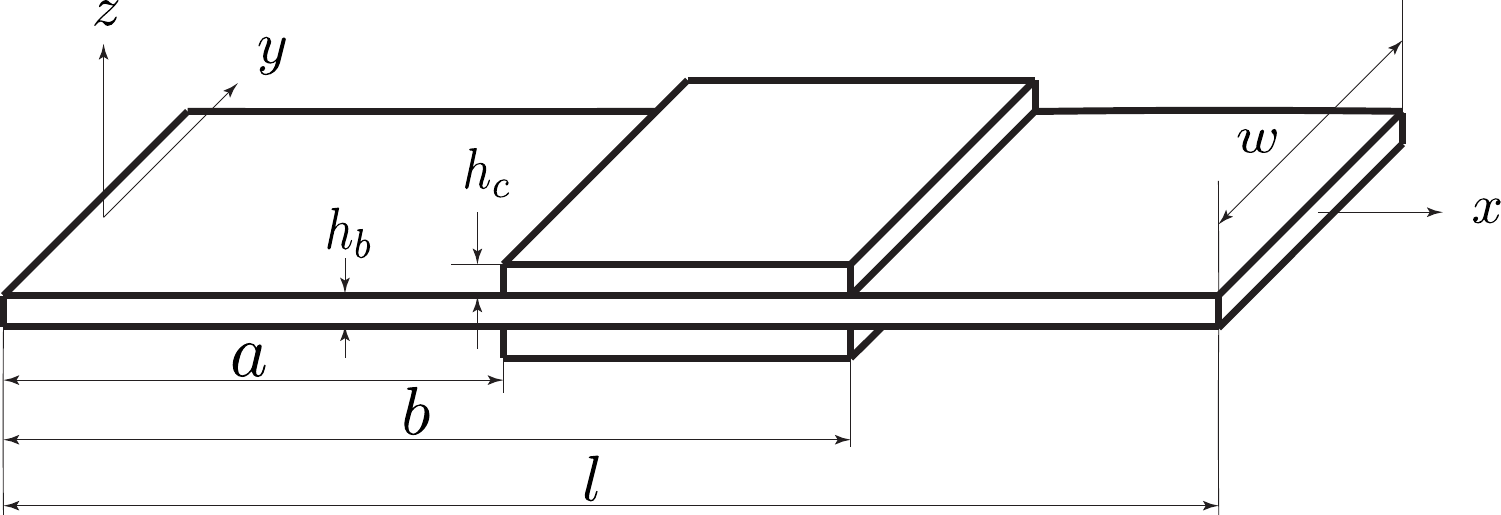}
\caption{Function 2D plot.}
\label{fig_pbeam}
\end{figure}

In this subsection, we derive the nonlinear equations of motion of the typical piezoelectric composite, the cantilevered bimorph, shown in Figure \ref{fig_pbeam}. The extended Hamilton's Principle states that of all the possible trajectories in the electromechanical configuration space, the actual motion satisfies the variational identity

\begin{equation}
\delta \int_{t_0}^{t_1} \left (  T - \mathcal{V}_{\mathcal{H}}  \right ) dt + \int_{t_0}^{t_1} \delta W dt = 0 \label{eq_ham_prin}
\end{equation}

with kinetic energy $T$, electromechanical potential $\mathcal{V}_\mathcal{H}$ defined below, electromechanical virtual work $\delta W$, initial time $t_0$, and final time $t_1$. The kinetic energy of the nonlinear piezoelectric beam is expressed as 
\begin{align}
    T = \frac{1}{2} \int_0^l m(x) (\dot{w} + \dot{\mathscr{z}})^2 dx = \frac{1}{2} \mathscr{m} \int_0^l (\dot{w} + \dot{\mathscr{z}})^2 dx
    \label{eq_kinEn}
\end{align}
with $m(x)$ the mass per unit length of the beam and $\mathscr{m}$ defined as $\mathscr{m} = \rho_s h_s + 2 \rho_p h_p$. In the above equation, $w = w(x,t)$ is the displacement from the neutral axis at location $x \in [0,l]$ at time $t$. The variable $\mathscr{z}(t)$ represents the displacement of the root of the beam, that is, it is the base motion that occurs in the $z$ direction defined relative to the beam. The terms $\rho$ and $h$ represent the density and thickness, respectively. The subscript $s$ represents the variables corresponding to the substrate and the subscript $p$ indicates those of the piezoceramic. The electric enthalpy for the nonlinear system is given by the relation
\begin{align*}
    \mathcal{V}_\mathcal{H} = \int_V \mathcal{H} dV = \int_{V_b} \mathcal{H} dV_b + \int_{V_p} \mathcal{H} dV_p.
\end{align*}

Substituting the expression for $\mathcal{H}$ in the above equation gives
\begin{align}
    \mathcal{V}_\mathcal{H} &= \int_{V_b} \frac{1}{2} C_b S_x^2 dV_b + \int_{V_c} \left( \frac{1}{2} C^{(0)}_{xx} S_x^2 + \frac{1}{3} C^{(1)}_{xx} S_x^3 + \frac{1}{4} C^{(2)}_{xx} S_x^4 \right.
    \notag
    \\
    & \hspace{3cm} \left. - \gamma_0 S_x E_z - \frac{1}{2} \gamma_1 S_x^2 E_z - \frac{1}{3} \gamma_2 S_x^3 E_z - \frac{1}{2} \nu_0 E_z^2 \right) dV_p
    \label{eq_elecenth}
\end{align}
with beam Young's modulus $C_b$, beam volume $V_b$ and piezoelectric patch volume $V_p$. We recall that the approximation for bending strain in Euler-Bernoulli beam theory is given by
\begin{align*}
    S_x(x,z,t) = -\frac{\partial^2 w(x,t)}{\partial x^2} z, \qquad \forall x \in [0,l], \qquad \forall z \in \left[ -\frac{h_b}{2} - h_c, \frac{h_b}{2} + h_c \right].
\end{align*}

Consider the term $\int_{V_p} \frac{1}{2} C^{(0)}_{xx} S_x^2 $ in the expression for electric enthalpy density. With the substitution of the expression for strain, we get
\begin{align*}
    \int_{V_p} \frac{1}{2} C_{xx}^{(0)} S_x^2 = \frac{1}{2} C_{xx}^{(0)} \int_{V_p} [(w'')^2 z^2] dV &= \frac{1}{2} C_{xx}^{(0)} \left( \int_{a}^{b} (w'')^2 dx \right) \left( \int_0^b dy \right) \left( 2 \int_{\frac{h_b}{2}}^{\frac{h_b}{2} + h_c} z^2 dz \right)
    \\
    &= 2 \underbrace{\left[ \frac{1}{6} C_{xx}^{(0)} b \left\{ \left( \frac{h_b}{2} + h_c \right)^3 - \left( \frac{h_b}{2} \right)^3 \right\} \right]}_{\coloneqq a_{(0,2)}} \int_{a}^{b} (w'')^2 dx
    \\
    &= 2 a_{(0,2)} \int_{a}^{b} (w'')^2 dx.
\end{align*}
The other terms in Equation \ref{eq_elecenth} can be simplified in a similar manner. The expression for electric enthalpy density after simplification has the form
\begin{align}
    \mathcal{V}_\mathcal{H} &= \frac{1}{2} C_b I_b \int_0^l (w'')^2 dx + 2 a_{(0,2)} \int_{a}^{b} (w'')^2 dx + 2 a_{(2,4)} \int_{a}^{b} (w'')^4 dx 
    \\
    & \hspace{3cm} + 2 b_{(1,1)} \left( \int_{a}^{b} w'' dx \right) E_z + 2 b_{(3,1)} \left( \int_{a}^{b} (w'')^3 dx \right) E_z - 2b_{(0,2)} E_z^2,
    \label{eq_elecPot}
\end{align}   

where we define

\begin{align*}
    a_{(0,2)} &:= \frac{1}{6} C_{xx}^{(0)} b \left[ \left( \frac{h_b}{2} + h_c \right)^3 - \left( \frac{h_b}{2} \right)^3 \right], \qquad a_{(2,4)} := \frac{1}{20} C_{xx}^{(2)} b \left[ \left( \frac{h_b}{2} + h_c \right)^5 - \left( \frac{h_b}{2} \right)^5 \right], \\
    b_{(0,2)} &:= \frac{1}{2} \nu_0 b h_c l_c, \qquad 
    b_{(1,1)} := \frac{1}{2} \gamma_0 b \left[ \left( \frac{h_b}{2} + h_c \right)^2 - \left( \frac{h_b}{2} \right)^2 \right], \\
    b_{(3,1)} &:= \frac{1}{12} \gamma_2 b \left[ \left( \frac{h_b}{2} + h_c \right)^4 - \left( \frac{h_b}{2} \right)^4 \right].
\end{align*}

For the time being, we omit the effects of damping in the following derivation. Following the details included in Appendix \ref{app_EOM}, the variational statement of Hamilton's principle yields the pair of equations

\begin{align}
    & m \ddot{w}
    + C_b I_b w'''' 
    + 4 a_{(0,2)} \left( \chi_{[a,b]} w'' \right)'' 
    + 8 a_{(2,4)} (\chi_{[a,b]} (w'')^3)'' 
    \notag
    \\
    & \hspace{4cm}
    + 2 b_{(1,1)} \chi_{[a,b]}'' E_z 
    + 6 b_{(3,1)} (\chi_{[a,b]} (w'')^2 )'' E_z
    = -m \ddot{\mathscr{z}}, 
    \label{eq_eom1}
    \\
    & 2b_{(1,1)} w'(b) - 2b_{(1,1)}w'(a) + 2b_{(3,1)} \left( \int_{a}^{b} (w'')^3 dx \right) + 4b_{(0,2)}E_z = 0,
    \label{eq_eom2}
\end{align}

where $\chi_{[a,b]}$ is the characteristic function of the interval $[a,b]$ defined as in Equation \ref{eq_chaeq}. These equations are subject to the corresponding variational boundary conditions

\begin{align*}
    \Big \{
    C_b I_b w'' + 4a_{(0,2)} \chi_{[a,b]} w'' + 8a_{(2,4)} \chi_{[a,b]} (w'')^3 
    \\
    + 2 b_{(1,1)} \chi_{[a,b]} E_z + 6 b_{(3,1)} \chi_{[a,b]} (w'')^2
    \Big \} \delta w' \big|_0^l &= 0, 
    \\
    \Big \{
    C_b I_b w''' + 4a_{(0,2)} \left(\chi_{[a,b]} w''\right)' + 8a_{(2,4)} \left( \chi_{[a,b]} (w'')^3 \right)'
    \\
    + 2 b_{(1,1)} \chi_{[a,b]}' E_z + 6 b_{(3,1)} \left( \chi_{[a,b]} (w'')^2 \right)'
    \Big \} \delta w \bigg|_0^l &= 0,
\end{align*}
and to the initial conditions $w(0) = w_0$ and $\dot{w}(0) = \dot{w}_0$. 

We know that the effects of nonlinearity in oscillators become most noticeable near the natural frequency. Hence, we approximate the solutions of Equations \ref{eq_eom1} and \ref{eq_eom2} using a single-mode approximation $w(x,t) = \psi(x) u(t)$. Following the detailed analysis in Appendix \ref{app_SMEOM}, the equations of motion are written
\begin{align}
    M \ddot{u}(t) + P \ddot{\mathscr{z}}(t) + \underbrace{[K_b + K_p]}_{K} u(t) + K_N u^3(t) + [B + Q_{N} u^2(t)] E_z = 0,
    \label{eq_SMEOM1}
    \\
    B u(t) + B_N u^3(t) = \mathcal{C} E_z,
    \label{eq_SMEOM2}
\end{align}
for constants $M, P, K_b, K_p, K_N, B, Q_n, B_N,$ and $\mathcal{C}$ defined in Appendix \ref{app_SMEOM}.

Note that the first equation defines the dynamics of the system and the second equation defines an algebraic relation between displacement and the electric field. From the second equation of motion, we get an expression for the electric field that has the form $E_z = [B u(t) + B_N u^3(t)]/\mathcal{C}$. Substituting this expression for electric field into the first equation of motion, we get

\begin{align*}
    M \ddot{u}(t) 
    + \underbrace{\left[ K + \frac{B^2}{\mathcal{C}} \right]}_{\hat{K}} u(t) 
    + \underbrace{\left[ K_N + \frac{B B_N}{\mathcal{C}} + \frac{Q_{N}B}{\mathcal{C}} \right]}_{\hat{K}_{N_1}} u^3(t) 
    + \underbrace{\frac{Q_{N} B_N}{\mathcal{C}}}_{\hat{K}_{N_2}} u^5(t) &= - P \ddot{\mathscr{z}}(t) \\
    M \ddot{u}(t) 
    + \hat{K} u(t) 
    + \hat{K}_{N_1} u^3(t) 
    + \hat{K}_{N_2} u^5(t) &= - P \ddot{\mathscr{z}}(t)
\end{align*}

After introducing a viscous damping term for the representation of energy losses, we have

\begin{align*}
    M \ddot{u}(t)
    + C \dot{u}(t)
    + \hat{K} u(t) 
    + \hat{K}_{N_1} u^3(t) 
    + \hat{K}_{N_2} u^5(t) &= - P \ddot{\mathscr{z}}(t).
\end{align*}

Let us define the state vector $\bm{x} = \{ x_1, x_2 \}^T = \{ u, \dot{u} \}^T$. Now, we can write the first order form of the governing equations as

\begin{align}
    \begin{Bmatrix}
    \dot{x}_1 \\ \dot{x}_2
    \end{Bmatrix}
    =
    \underbrace{\begin{bmatrix}
    0 & 1 \\
    -\frac{\hat{K}}{M} & -\frac{C}{M}
    \end{bmatrix}}_{A}
    \begin{Bmatrix}
    x_1 \\ x_2
    \end{Bmatrix}
    +
    \underbrace{\begin{Bmatrix}
    0 \\ -\frac{P}{M}
    \end{Bmatrix}}_{B} \underbrace{\ddot{\mathscr{z}}(t)}_{\mathscr{u}(t)}
    +
    \underbrace{\begin{Bmatrix}
    0 \\ 1
    \end{Bmatrix}}_{B_N}
    \underbrace{ 
    \left(
    - \frac{\hat{K}_{N_1}}{M} x_1^3(t) - \frac{\hat{K}_{N_2}}{M} x_1^5(t) 
    \right)
    }_{f(\bm{x}(t))},
    \label{eq_piezomodel}
\end{align}
or 
\begin{align*}
    \dot{\bm{x}}(t) = A \bm{x}(t) + B \mathscr{u}(t) + B_N f(\bm{x}(t)).
\end{align*}
We make several observations before proceeding to the adaptive estimation problem treated in the next section. 
Note that the specific form of function $f(\bm{x}) = f(x_1)$ that is given in Equation 10 above has been constructed assuming the only unknown terms are the nonlinearities arising from the constitutive laws. We allow for a wider class of uncertainty that can be expressed as $f(\bm{x}) = f(x_1, x_2)$. For instance, if the viscous damping coefficient is uncertain or unknown, the damping term should be subsumed into $f(x_1, x_2)$. With these considerations in mind, the derivations in the next section are carried out for the more general case when $f = f(x_1,x_2)$. However, when we prepare finite-dimensional approximations in Section \ref{subsec_finapprox} for the simulations in Sections \ref{sec_imp} and \ref{sec_results}, we specialize examples to the case $f = f(x_1)$ described above.

%
\section{Adaptive Estimation in RKHS}
\label{sec_adap}
In this section, we pose the estimation problem for the approximation of the unknown nonlinear function $f$ and review the theory of RKHS adaptive estimation. The governing equation of the plant, the piezoelectric oscillator modeled in Section \ref{sec_model}, has the general form

\begin{align}
    \dot{\bm{x}}(t) = A \bm{x}(t) + B \mathscr{u}(t) + B_N f(\bm{x}(t)).
    \label{eq_plantModel}
\end{align}

We denote the state space of this evolution law by $X = \mathbb{R}^d$, so that $\bm{x}(t) \in X$. Under the assumption of full state observability, the problem of estimation of the states $\bm{x}(t)$ at a given time instant $t$ is a classical state estimation problem. However, the problem of interest in this paper is the estimation of the unknown function $f$. Problems of this type generally involve the definition of an estimator system that evolves in parallel with the actual plant. The model of the estimator for the plant defined by Equation \ref{eq_plantModel} is taken in the form 
\begin{align}
    \dot{\hat{\bm{x}}}(t) = A \hat{\bm{x}}(t) + B \mathscr{u}(t) + B_N \hat{f}(t,\bm{x}(t)).
    \label{eq_estModel}
\end{align}

In Equation \ref{eq_estModel}, note that the estimate $\hat{f}$ of the function $f$ depends not only on the actual (measured) states $\bm{x}(t)$ but also the time $t$. We want the function estimate $\hat{f}(t,\cdot)$ to converge in time to the actual function $f(\cdot)$ in some suitable function space norm as $t \to \infty$. 

In addition to the estimator model, it is also important to define the hypothesis space, the space of functions in which the function $f$ and the function estimate $\hat{f}$ live. In this paper, we assume that the unknown nonlinear function $f$ lives in the infinite dimensional RKHS $\mathcal{H}$ equipped with the reproducing kernel $\mathcal{K}_X: X \times X \to \mathbb{R}$. Recall that the reproducing property of the kernel states that, for any $\bm{x} \in X$ and $f \in \mathcal{H}_X$, $\left( \mathcal{K}(\bm{x},\cdot),f \right)_{\mathcal{H}_X} = f(\bm{x})$. It is well known that the existence of a reproducing kernel guarantees the boundedness of the evaluation functional $\mathcal{E}_{\bm{x}}: \mathcal{H} \to \mathbb{R}$, which is defined by the condition that $\mathcal{E}_{\bm{x}} f = (\mathcal{K}(\bm{x},\cdot),f)_\mathcal{H}$. In this paper, we restrict to RKHS in which the reproducing kernel is bounded by a constant. This implies that the injection $i : \mathcal{H} \to C(\Omega)$ from the RKHS $\mathcal{H}$ to the space of continuous function on $\Omega$, $C(\Omega)$, is uniformly bounded \cite{Bobade2019}. This fact is used to prove the existence and uniqueness of the solution of the error system. A more detailed discussion about RKHS can be found in \cite{Aronszajn1950, Berlinet2011, Muandet2017}.

In addition to the estimator model, we also need an equation that defines the evolution (time derivative) of the function estimate. This is given by the learning law
\begin{align}
    \dot{\hat{f}}(t) = \Gamma^{-1} (B_N \mathcal{E}_{\bm{x}(t)})^* P (\bm{x}(t) - \hat{\bm{x}}(t)),
    \label{eq_lLaw}
\end{align}

where $\Gamma \in \mathbb{R}^+$, $\mathcal{E}_{\bm{x}}$ is the evaluation functional at $\bm{x} \in X$, and the notation $(\cdot)^*$ denotes the adjoint of an operator. Further, the matrix $P \in \mathbb{R}^{d \times d}$ is the symmetric positive definite solution of the Lyapunov's equation $A^T P + PA = - Q$, where $Q \in \mathbb{R}^{d \times d}$ is an arbitrary but fixed symmetric positive definite matrix. 

The existence and uniqueness of a solution for the estimator models given by Equations \ref{eq_estModel} and \ref{eq_lLaw} can be proved under the assumption that the excitation input is continuous and we are working in an uniformly embedded RKHS as mentioned above. The following theorem proves this statement. 
\begin{theorem}
Define $\mathbb{X}:= \mathbb{R}^d \times \mathcal{H}$, and suppose that $\bm{x} \in C([0,T]; \mathbb{R}^d)$, $\mathscr{u} \in C([0,T]; \mathbb{R})$ and that the embedding $i: \mathcal{H} \hookrightarrow C(\Omega)$ is uniform in the sense that there is a constant $C>0$ such that for any $f \in \mathcal{H}$,
\begin{align*}
    \|f\|_{C(\Omega)} \equiv \|if\|_{C(\Omega)} \leq C \|f\|_{\mathcal{H}}.
\end{align*}
Then for any $T > 0$, there is a unique mild solution $(\hat{\bm{x}},\hat{f}) \in C([0,T],\mathbb{X})$ to 
\begin{align}
    \begin{Bmatrix}
    \dot{\hat{\bm{x}}}(t) \\ \dot{\hat{f}}(t)
    \end{Bmatrix}
    = 
    \begin{Bmatrix}
    A \hat{\bm{x}}(t) + B \mathscr{u}(t) + B_N \mathcal{E}_{\bm{x}(t)} \hat{f}(t) \\
    \Gamma^{-1} (B_{N} \mathcal{E}_{\bm{x}(t)})^* P (\bm{x}(t) - \hat{\bm{x}}(t))
    \end{Bmatrix},
    \label{eq_EstMod}
\end{align}
and the map $\hat{X}_0 \equiv (\hat{\bm{x}}_0,\hat{f}_0) \mapsto (\hat{\bm{x}},\hat{f}$) is Lipschitz continuous from $\mathbb{X}$ to $C([0,T],\mathbb{X})$.
\label{thm_1}
\end{theorem}

\begin{proof}
We set $X(t):= (\hat{x}(t),\hat{f}(t)) \in \mathbb{X}$. Equation \ref{eq_EstMod} given above can be rewritten as
\begin{align}
    \begin{Bmatrix}
    \dot{\hat{\bm{x}}}(t) \\ \dot{\hat{f}}(t)
    \end{Bmatrix}
    &= 
    \underbrace{\begin{bmatrix}
    A & 0  \\
    0 & A_0
    \end{bmatrix}}_{\mathcal{A}}
    \begin{Bmatrix}
    \hat{\bm{x}}(t) \\ \hat{f}(t)
    \end{Bmatrix}
    +
    \underbrace{\begin{Bmatrix}
    B \mathscr{u}(t) + B_N \mathcal{E}_{\bm{x}(t)} \hat{f}(t)
    \\
    -A_0 \hat{f}(t) + \Gamma^{-1} (B_{N} \mathcal{E}_{\bm{x}(t)})^* P (\bm{x}(t) - \hat{\bm{x}}(t))
    \end{Bmatrix}}_{\mathcal{F}(t,X(t))},
    \label{eq_EstMod2}
    \\
    \begin{Bmatrix}
    \hat{\bm{x}}(t_0) \\
    \hat{f}(t_0)
    \end{Bmatrix}
    &=
    \begin{Bmatrix}
    \hat{\bm{x}}_0 \\
    \hat{f}_0
    \end{Bmatrix}, \notag
\end{align}
where $-A_0$ is an arbitrary bounded linear operator from $\mathcal{H}$ to $\mathcal{H}$. It is clear from the above equation that $\mathcal{A}$ is a bounded linear operator. We know that every bounded linear operator is the infinitesimal generator of a $C_0$-semigroup on $\mathbb{X} := \mathbb{R}^d \times \mathcal{H}$ (Theorem 1.2, Chapter 1 of \cite{Pazy2012}). Now, consider the function $\mathcal{F}$. For each $t \geq 0$, we have

\begin{align*}
    \| \mathcal{F}(t,\hat{X}) - \mathcal{F}(t,\hat{Y}) \| = \left \| 
    \begin{Bmatrix}
    B_N \mathcal{E}_{\bm{x}(t)} (\hat{f}_{\hat{x}}(t) - \hat{f}_{\hat{y}} (t)) 
    \\
    -A_0 (\hat{f}_{\hat{x}}(t) - \hat{f}_{\hat{y}} (t)) + \Gamma^{-1} (B_N \mathcal{E}_{\bm{x}(t)})^* P (\hat{y}(t) - \hat{x}(t))
    \end{Bmatrix}
    \right \|
    \leq D \| \hat{X} - \hat{Y} \|,
\end{align*}
where $\hat{X}:= (\hat{x}, \hat{f}_{\hat{x}})$, $\hat{Y}:= (\hat{y}, \hat{f}_{\hat{y}})$, and $D \geq 0$ is a constant. Note that we are able to achieve the above bound because of uniform boundedness of the evaluation functional $\mathcal{E}_{x(t)}$. Thus, for each $t \geq 0$, the map $\hat{X} \mapsto \mathcal{F}(t,\hat{X})$ is uniformly globally Lipschitz continuous. We also note that the map $t \mapsto \mathcal{F}(t,\hat{X})$ is continuous for each $\hat{X} \in \mathbb{X}$ since $\mathscr{u}$ is continuous. Using Theorem 1.2 in Chapter 6 of \cite{Pazy2012}, we can conclude that the above initial value problem has a unique mild solution, and the map $\hat{X}_0 \equiv (\hat{\bm{x}}_0,\hat{f}_0) \mapsto (\hat{\bm{x}},\hat{f}$) is Lipschitz continuous from $\mathbb{X}$ to $C([0,T],\mathbb{X})$. 
\end{proof}

Suppose that $\tilde{\bm{x}}(t) :=\bm{x}(t) - \hat{\bm{x}}(t)$ and $\tilde{f}(t,\cdot):= f(\cdot) - \hat{f}(t,\cdot)$ denote the state error and the function error, respectively. Equations \ref{eq_plantModel}, \ref{eq_estModel} and \ref{eq_lLaw} can now be expressed in terms of the error equation

\begin{align}
    \begin{Bmatrix}
    \dot{\tilde{\bm{x}}}(t) \\ \dot{\tilde{f}}(t)
    \end{Bmatrix}
    = 
    \begin{bmatrix}
    A & B_N \mathcal{E}_{\bm{x}(t)} \\
    - \Gamma^{-1} (B_{N} \mathcal{E}_{\bm{x}(t)})^* P & 0
    \end{bmatrix}
    \begin{Bmatrix}
    \tilde{\bm{x}}(t) \\ \tilde{f}(t)
    \end{Bmatrix}.
    \label{eq_errEst}
\end{align}

Note, the above equation evolves in $\mathbb{R}^d \times \mathcal{H}$. Also, even though the original Equation \ref{eq_plantModel} and the estimator Equation \ref{eq_estModel} are not the same as in Reference \cite{Bobade2019}, the above error equation does have the same form as that studied in \cite{Bobade2019}. The existence and uniqueness of a solution for this equation are given by the following theorem.

\begin{theorem}
Define $\mathbb{X}:= \mathbb{R}^d \times \mathcal{H}$, and suppose that $\bm{x} \in C([0,T]; \mathbb{R}^d)$ and that the embedding $i: \mathcal{H} \hookrightarrow C(\Omega)$ is uniform in the sense that there is a constant $C>0$ such that for any $f \in \mathcal{H}$,
\begin{align*}
    \|f\|_{C(\Omega)} \equiv \|if\|_{C(\Omega)} \leq C \|f\|_{\mathcal{H}}.
\end{align*}
Then for any $T > 0$, there is a unique mild solution $(\tilde{\bm{x}},\tilde{f}) \in C([0,T],\mathbb{X})$ to Equations \ref{eq_errEst} and the map $X_0 \equiv (\tilde{\bm{x}}_0,\tilde{f}_0) \mapsto (\tilde{\bm{x}},\tilde{f})$ is Lipschitz continuous from $\mathbb{X}$ to $C([0,T],\mathbb{X})$.
\end{theorem}

The proof for this theorem is very similar to the proof of Theorem \ref{thm_1} and is given in \cite{Bobade2019}. Note that the above theorem does not study the stability nor the asymptotic stability of the error system. In other words, the convergence of the state error and the function error to the origin is not addressed by this theorem. This aspect is addressed in the following subsection.

\subsection{Persistence of Excitation}
\label{subsec_pe}
The convergence of state and function errors is guaranteed by additional conditions, commonly referred to as the persistence of excitation (PE) conditions \cite{Ioannou, Sastry2011, Narendra2012}. These have been extended to the RKHS framework in \cite{jia2020a,jia2020b}. This section reviews the persistence of excitation conditions for adaptive estimators on RKHS in detail.

Before taking a look at the PE conditions for the adaptive estimator in the RKHS, it is important to note that they are defined over a set $\Omega \subseteq X$. Now, we can define $\mathcal{H}_\Omega := \overline{ \{ \mathcal{K}(\bm{x},\cdot) | \bm{x} \in \Omega \} }$. Note that $\mathcal{H}_\Omega$ is subspace of $\mathcal{H}_X$. The following definitions give us two closely related versions of the PE condition on the set $\Omega$.

\begin{definition}
\label{def_PE1}
The trajectory $\bm{x}: t \mapsto \bm{x}(t) \in \mathbb{R}^d$ persistently excites the indexing set $\Omega$ and the RKHS $\mathcal{H}_\Omega$ provided there exist positive constants $T_0, \gamma, \delta,$ and $\Delta$, such that for each $t \geq T_0$ and any $g \in \mathcal{H}_\Omega$ with $\|g\|_{H_\Omega} = 1$, there exists exists $s \in [t,t+\Delta]$ such that
\begin{align*}
    \left| \int_{s}^{s+\delta} \mathcal{E}_{\bm{x}(\tau)} g d \tau \right| \geq \gamma > 0.
\end{align*}
\end{definition}

\begin{definition}
\label{def_PE2}
The trajectory $\bm{x}: t \mapsto \bm{x}(t) \in \mathbb{R}^d$ persistently excites the indexing set $\Omega$ and the RKHS $H_\Omega$ provided there exists positive constants $T_0$, $\gamma$, and $\Delta$, such that 
\begin{align*}
    \int_t^{t+\Delta} \left( \mathcal{E}^*_{\bm{x}(\tau)} \mathcal{E}_{\bm{x}(\tau)}g,g \right)_{H_\Omega} d \tau \geq \gamma > 0
\end{align*}
for all $t \geq T_0$ and any $g \in H_\Omega$ with $\| g \|_{H_\Omega} = 1$.
\end{definition}

Notice that both the PE conditions are defined on the set $\Omega$. It would be ideal if $\Omega = X$, the space on which the nonlinear function is defined. However, in most practical applications, the set $\Omega$ is a subset of the state space $X$. The following theorem relates both the PE conditions given above.

\begin{theorem}
The PE condition in Definition PE. \ref{def_PE1} implies the one in Definition PE. \ref{def_PE2}. Further, if the family of functions defined by $\{ g(\bm{x}(\cdot)): t \mapsto g(\bm{x}(t)) : \| g \|_{H_\Omega} = 1, g \in H_\Omega \}$ is uniformly equicontinuous, then the PE condition in Definition PE. \ref{def_PE2} implies the one in Definition PE. \ref{def_PE1},
\end{theorem}

With the PE conditions defined, the following theorem addresses the convergence of the states of the error system to the origin.

\begin{theorem}
Suppose the trajectory $\bm{x} : t \mapsto \bm{x}(t)$ persistently exists the RKHS $\mathcal{H}_\Omega$ in the sense of Definition PE. \ref{def_PE1}. Then the estimation error system in Equation \ref{eq_errEst} is uniformly asymptotically stable at the origin. In particular, we have
\begin{align*}
    \lim_{t \to \infty} \| \tilde{\bm{x}}(t) \| = 0, \hspace{1in} \lim_{t \to \infty} \| \tilde{f}(t) \|_{\mathcal{H}_\Omega} = 0.
\end{align*}
\end{theorem}
The proof for this theorem can be found in \cite{Guo2019}. Intuitively, the second PE condition implies that the state trajectory should repeatedly enter every neighborhood of all the points in the set $\Omega$ infinitely many times. To satisfy this, it makes sense to pick the set $\Omega$ to be the positive limit set $\omega^+(\bm{x}_0)$ or one of its subsets. The following theorem from \cite{Kurdila2019PE} affirms that the persistently excited sets are in fact contained in the positive limit set.
\begin{theorem}
Let $H_X$ be the RKHS of functions over $X$ and suppose that this RKHS includes a rich family of bump functions. If the PE condition in Definition PE. \ref{def_PE2} holds for $\Omega$, then $\Omega \subseteq \omega^+(\bm{x}_0)$, the positive limit set corresponding to the initial condition $\bm{x}_0$.
\end{theorem}

%

\subsection{Finite Dimensional Approximation}
\label{subsec_finapprox}
As mentioned above, the evolution of the error equation and the learning law for the RKHS adaptive estimator is in  $\mathbb{R} \times \mathcal{H}$. In essence, the learning law constitutes a distributed parameter system since $\tilde{f}(t)$ evolves in a infinite-dimensional space. Thus, to implement this adaptive estimator, the persistently excited infinite-dimensional space $\mathcal{H}_{\Omega}$ is approximated by a nested, dense collection $\{ \mathcal{H}_n \}_{n \in \mathbb{N}}$ of finite-dimensional subspaces. Recall that even though the particular nonlinear function $f$ based on the choice of constitutive nonlinearities in Equation \ref{eq_eed} is a function $f = f(x_1)$, we have elected to cast the problem in terms of the more general nonlinear function $f = f(x_1,x_2)$. In this section, we will continue with the analysis of finite-dimensional approximation for the more general unknown nonlinear function $f = f(x_1,x_2)$, which results in Equations \ref{eq_llapprox1} and \ref{eq_llapprox2} below. Modifications of these equations to study the particular case in which $f = f(x_1)$ are straightforward, and we summarize this specific case at the beginning of Section \ref{sec_results}. We leave the details to the reader. Let $\Pi_n$ represent the projection operator from infinite-dimensional $\mathcal{H}_{\Omega}$ to the finite-dimensional $\mathcal{H}_n$. Now, the finite-dimensional approximations of the adaptive estimator equations can be expressed as
\begin{align}
    \dot{\hat{\bm{x}}}_n(t) &= A \hat{\bm{x}}_n(t) + B \mathscr{u}(t) + B_{N} \mathcal{E}_{\bm{x}(t)} \Pi_n^* \hat{f}_n(t), 
    \label{eq_llapprox1}
    \\
    \dot{\hat{f}}_n(t) &= \Gamma^{-1} \left( B_{N} \mathcal{E}_{\bm{x}(t)} \Pi_n^* \right)^* P \tilde{\bm{x}}_n(t),
    \label{eq_llapprox2}
\end{align}
where $ \tilde{\bm{x}}_n:= \bm{x} - \hat{\bm{x}}_n$.
%
\begin{theorem}
Suppose that $\bm{x} \in C([0,T],\mathbb{R}^d)$ and that the embedding $i : \mathcal{H} \hookrightarrow C(\Omega)$ is uniform in the sense that 
\begin{align*}
    \|f\|_{C(\Omega)} \equiv \|if\|_{C(\Omega)} \leq C \|f\|_{\mathcal{H}}.
\end{align*}
Then for any $T>0$,
\begin{align*}
    \| \hat{\bm{x}} - \hat{\bm{x}}_n \|_{C([0,T];\mathbb{R}^d)} &\to 0, \\
    \|\hat{f} - \hat{f}_n \|_{C([0,T];\mathbb{R}^d)} &\to 0,
\end{align*}
as $n \to \infty$.
\end{theorem}

The proof of the above theorem can be found in \cite{Bobade2019}. As noted earlier, the estimator equations considered in \cite{Bobade2019} are different from the ones considered in this paper. However, the error equations for $\hat{\bm{x}} - \hat{\bm{x}}_n$ and $\hat{f} - \hat{f}_n$ still have the same form as in \cite{Bobade2019}, and the proof of the above theorem will remain the same. 

%

\section{RKHS Adaptive Estimator Implementation}
\label{sec_imp}
\begin{figure}[htb!]
\centering
\includegraphics[scale = 0.4]{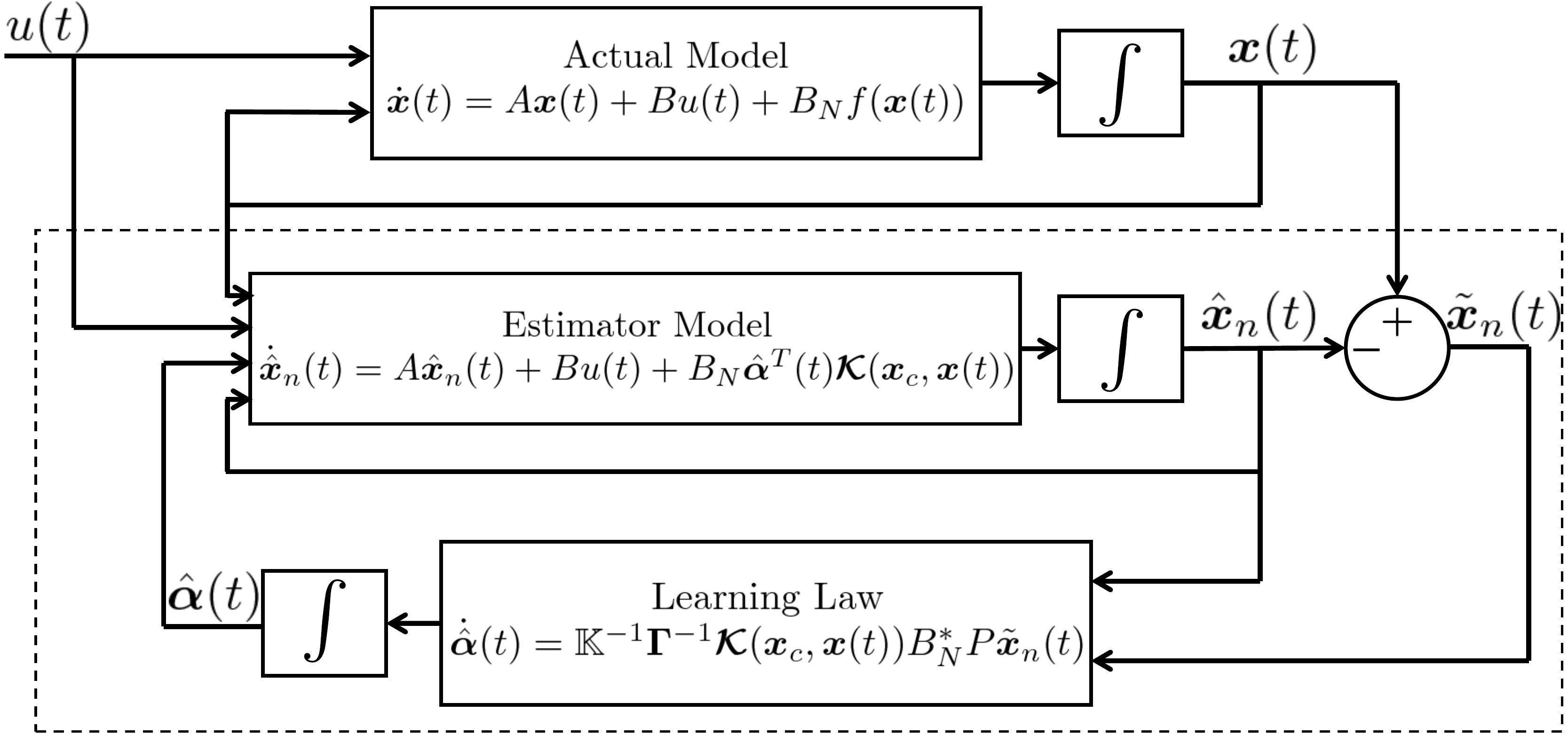}
\caption{Adaptive Parameter Estimator Block Diagram.}
\label{fig_bd}
\end{figure}

The previous section discussed the theory behind estimators that evolve in an RKHS. This section presents the algorithm for the implementation of the theory. Figure \ref{fig_bd} shows the block diagram of the adaptive estimator. The actual model shown in the figure corresponds to the true system excited by the input $\mathscr{u}$, and we assume that we can measure all the states $\bm{x}(t)$ of this true system. The estimator and learning law blocks in the diagram are what we implement on the computer. Let us first take a look at the estimator model. The operator $\Pi_n^*$ in the estimator model is the adjoint of the orthogonal projection/approximation operator $\Pi_n$. It is equivalent to the inclusion map that maps an element of $\mathcal{H}_n$ space to the same element in the $\mathcal{H}_{\Omega}$ space. Thus, the term $\mathcal{E}_{x(t)} \Pi_n^* \hat{f}_n(t)$ in the estimator model is the same as $\mathcal{E}_{\bm{x}(t)} \hat{f}_n(t) = \hat{f}_n(t,\bm{x}(t))$.

Now, let us take a look at the learning law given in Equation \ref{eq_llapprox2}. It is a derivative of a function, and we cannot directly implement it on a computer. To convert it to a form that is solvable using numerical methods, we take the inner product of the learning law with $\mathcal{K}(\bm{x}_i,\cdot)$. Before proceeding with this step, let us recall that the finite-dimensional function estimate $\hat{f}_n (t,\cdot)$ can be expressed as $\hat{f}_n (t,\cdot) = \sum_{j=1}^n \hat{\alpha}_j(t) \mathcal{K}(\bm{x}_j, \cdot) = \hat{\bm{\alpha}}^T(t) \bm{\mathcal{K}}(\bm{x}_c, \cdot) $. Thus, for $i = 1,\ldots,n$,

\begin{align*}
    \left( \mathcal{K}(\bm{x}_i,\cdot),\dot{\hat{f}}_n (t) \right)_{\mathcal{H}_X} 
    &= \left( \mathcal{K}(\bm{x}_i,\cdot), \Gamma^{-1} \left( B_{N} \mathcal{E}_{\bm{x}(t)} \Pi_n^* \right)^* P \tilde{\bm{x}}_n(t) \right)_{\mathcal{H}_X}
    \\
    \implies
    \sum_{j=1}^n \mathcal{K}(\bm{x}_i,\bm{x}_j) \dot{\hat{\alpha}}_j(t) 
    &= \Gamma^{-1} \left( B_{N} \mathcal{E}_{\bm{x}(t)} \mathcal{K}(\bm{x}_i,\cdot), P \tilde{\bm{x}}_n(t) \right)_{\mathcal{H}_X}.
\end{align*}
Thus, if $\hat{\bm{\alpha}}(t):= \{ \hat{\alpha}_1(t),\ldots,\hat{\alpha}_n(t) \}^T$, its time derivative is given by the expression
\begin{align*}
    \dot{\hat{\bm{\alpha}}}(t)= \mathbb{K}^{-1} \bm{\Gamma}^{-1} \bm{\mathcal{K}}(\bm{x}_{c},\bm{x}(t)) B_{N}^* P \tilde{\bm{x}}_n(t),
\end{align*}
where $\mathbb{K}$ is the symmetric positive definite Grammian matrix whose $ij^{th}$ element is defined as $\mathbb{K}_{ij} := \mathcal{K}(\bm{x}_i,\bm{x}_j)$, $\bm{\Gamma}:= \Gamma \mathbb{I}_{n}$ is the gain matrix, and $\bm{\mathcal{K}}(\bm{x}_{c},\bm{x}(t)):= \{ \mathcal{K}(\bm{x}_1,\bm{x}(t)),\ldots,\mathcal{K}(\bm{x}_n,\bm{x}(t)) \}^T$.

The above equation gives us an expression for the rate at which the coefficients of the kernels change with time. Therefore, the implementation of the adaptive estimator amounts to integration of the equations

\begin{align}
    \dot{\hat{\bm{x}}}_n(t) &= A \hat{\bm{x}}_n(t) + B \mathscr{u}(t) + B_{N} \hat{\bm{\alpha}}^T(t) \bm{\mathcal{K}}(\bm{x}_c, \bm{x}(t)),
    \label{eq_eqimp1}
    \\
    \dot{\hat{\bm{\alpha}}}(t) &= \mathbb{K}^{-1} \bm{\Gamma}^{-1} \bm{\mathcal{K}}(\bm{x}_{c},\bm{x}(t)) B_{N}^* P \tilde{\bm{x}}_n(t).
    \label{eq_eqimp2}
\end{align}

From the discussion in Subsection \ref{subsec_pe}, it is clear that the persistence of excitation is sufficient to ensure parameter convergence. However, it is hard and sometimes impossible to check if a given space is persistently exciting. The following theorem from \cite{Kurdila1995} gives us a sufficient condition for the persistence of excitation that is easy to verify. However, this theorem is only applicable to cases where radial basis functions over $\mathbb{R}^d$ generate the RKHS. Furthermore, we can only use this sufficient condition to check the persistence of excitation of finite-dimensional spaces. However, since all implementation is in the finite-dimensional spaces, the following theorem provides us a powerful tool to verify the convergence of parameters in practical applications.

\begin{theorem}
Let $\epsilon < \frac{1}{2} \min_{i \neq j} \| \bm{x}_i - \bm{x}_j \|$, where $\bm{x}_i$ and $\bm{x}_j$ are the kernel centers $\{\bm{x}_1, \ldots, \bm{x}_n \}$. For every $t_0 \geq 0$ and $\delta > 0$, define
\begin{align*}
    I_i = \{ t \in [t_0, t_0 + \delta] : \| \bm{x}(t) - \bm{x}_i \| \leq \epsilon \}.
\end{align*}
If there exists a $\delta$ such that the measure of $I_i$ is bounded below by a positive constant that is independent of $t_0$ and the kernel center $\bm{x}_i$, and if the measure of $[t_0,t_0+\delta]$ less than or equal to $\delta$, then the space $\mathcal{H}_n$ is persistently exciting.
\label{thm_sc}
\end{theorem}

We have to note that the persistence of excitation of $H_n$ does not imply the convergence of error to $0$ since the function $f$ belongs to the infinite-dimensional space $\mathcal{H}$. However, it can be shown that the limit of the error norm $\lim_{t \to \infty} \| f - \hat{f}_n(t) \|_{\mathcal{H}_\Omega}$ is bounded above by a positive constant. We refer the reader to $\cite{Kurdila1995}$ for a more detailed discussion on the convergence of parameters in finite-dimensional spaces.

The following algorithm gives a step by step procedure for implementing the RKHS adaptive estimator.
%
\begin{algorithm}[H]
\KwIn{$\bm{x}(t)$,$w^+(\bm{x}_0)$}
\KwOut{$\hat{f}_n (T,\cdot)$}
Choose the RKHS $\mathcal{H}_{X}$ and the corresponding reproducing kernel $\mathcal{K}(\cdot,\cdot)$.
\\
Choose kernel centers $\bm{x}_i$, for $i=1,\ldots,N$ uniformly distributed on $w^+(\bm{x}_0)$,
\linebreak
if $X$ is equal to the state space, choose kernels centers on $w^+(\bm{x}_0)$,
\linebreak
if $X$ is a proper subset of the state space, choose kernel centers on the projection of $w^+(\bm{x}_0)$ on to the space $X$.
\\
Run the adaptive estimator until the parameters converge. \linebreak
Integrate 
\begin{align*}
    \dot{\hat{\bm{x}}}_n(t) &= A \hat{\bm{x}}_n(t) + B \mathscr{u}(t) + B_{N} \hat{\bm{\alpha}}^T(t) \bm{\mathcal{K}}(\bm{x}_c, \bm{x}(t)),
    \\
    \dot{\hat{\bm{\alpha}}}(t) &= \mathbb{K}^{-1} \bm{\Gamma}^{-1} \bm{\mathcal{K}}(\bm{x}_{c},\bm{x}(t)) B_{N}^* P \tilde{\bm{x}}_n(t)
\end{align*}
over the interval $[0,T]$.
\\
Define $\hat{f}_n(T,\cdot) := \hat{\bm{\alpha}}^T(T) \bm{\mathcal{K}}(\bm{x}_c, \cdot)$.
\label{algo_rkhsAdapEst}
\caption{RKHS adaptive estimator implementation}
\end{algorithm}
\section{Numerical Simulation Results}
\label{sec_results}
In this section, we consider the prototypical piezoelectric oscillator example modeled in Section \ref{sec_model} to study the effectiveness of an RKHS adaptive estimator and make qualitative studies of convergence. As emphasized above, the finite-dimensional Equations \ref{eq_llapprox1} and \ref{eq_llapprox2} are stated for the general analysis when the unknown function $f = f(x_1,x_2)$. In this section, we study qualitative convergence properties in the specific case that $f = f(x_1)$. For this specific example, it is straightforward to show that the finite-dimensional equations have the form
\begin{align*}
    \dot{\hat{\bm{x}}}_n(t) &= A \hat{\bm{x}}_n(t) + B \mathscr{u}(t) + B_{N} \mathcal{E}_{x_1(t)} \Pi_n^* \hat{f}_n(t),
    \\
    \dot{\hat{f}}_n(t) &= \Gamma^{-1} \left( B_{N} \mathcal{E}_{x_1(t)} \Pi_n^* \right)^* P \tilde{\bm{x}}_n(t).
\end{align*}
These equations evolve in $\mathbb{R}^d \times \mathcal{H}_n$, where $\mathcal{H}_n = span \{ \mathcal{K}(x_{1,i}, \cdot ) \}$ is defined in terms of the kernel on $\mathbb{R}$ $\mathcal{K}: \mathbb{R} \times \mathbb{R} \to \mathbb{R}$ and displacement samples $\bm{x}_c = \{x_{1,i} \}_{i=1}^n = \Omega_n \subseteq \Omega \subseteq \mathbb{R}$. With this interpretation and the definition $\mathcal{K}(\bm{x}_c, \bm{x}(t)) := \{ \mathcal{K}(x_{1,1},x_1(t)), \ldots, \mathcal{K}(x_{1,n},x_1(t)) \}^T$, the specific governing equations still have the form shown in Equations \ref{eq_eqimp1}, \ref{eq_eqimp2}, and Algorithm \ref{algo_rkhsAdapEst} applies. Tables \ref{table_simprop1} and \ref{table_simprop2} list the numerical values of the parameters used to build the actual model shown in Equation \ref{eq_piezomodel}. We used the shape function corresponding to the first cantilever beam mode while modeling the system to get Equations \ref{eq_SMEOM1} and \ref{eq_SMEOM2}. Table \ref{table_simprop2} also shows the input used to drive the actual system.

\begin{table}[htb!]
\centering
    \footnotesize
    \caption{Parameters of the actual system used in simulations}
    \begin{subtable}{0.4\linewidth}
    \centering
    \subcaption{Piezoceramic parameters}
    \begin{tabular}{c|c|c}
    \toprule
    & Parameter & Value \\
    \hline 
    \hline
    \multirow{12}{*}{Piezoceramic (PIC 151)} & $\rho_p$ & $7790$ (kg/m$^3$)\\
    & $h_p$ & 0.001 (m)\\
    & $a$ & 0\\
    & $b$ & $l$\\
    & $d_{31,0}$ & -2.1e-10 (m/V) \\
    & $d_{31,1}$ & -36.9746 (m/V) \\
    & $d_{31,2}$ & -0.03596 (m/V) \\
    & $E_{p0}$ & 0.667e+11 (Pa) \\
    & $E_{p1}$ & -3.328e-12 (Pa) \\
    & $E_{p2}$ & -1.4e+18 (Pa) \\
    & $\epsilon_{33}$ & 2.12e-8 (F/m)\\
    \bottomrule
    \end{tabular}
    \label{table_simprop1}
    \end{subtable}
    \hspace{1cm} 
    \begin{subtable}{0.4\linewidth}
    \centering
    \subcaption{Other parameters}
    \begin{tabular}{c|c|c}
    \toprule
    & Parameter & Value \\
    \hline 
    \hline
    \multirow{6}{*}{Substrate} & Material & St 37 \\
    & $\rho_b$ & 7800 (kg/m$^3$) \\
    & $C_b$ & 2.089e+11 (Pa) \\
    & $l$ & 0.4 (m)\\
    & $b$ & 0.025 (m)\\
    & $h$ & 0.003 (m)\\
    \hline
    \multirow{2}{*}{Damping} & $\alpha$ & 0.1\\
    & $\beta$ & 1e-3\\
    \hline
    \multirow{3}{*}{Input} & $\mathscr{u}$ & $A \sin(\omega t)$ \\
    & Amplitude $A$ & 1 (m/s$^2$)\\
    & Frequency $\omega$ & 22.5 (rad/s)\\
    \bottomrule
    \end{tabular}
    \label{table_simprop2}
    \end{subtable}
    \label{table_simprop}
\end{table}

\begin{figure}[htb!]
\centering
\includegraphics[scale = 0.55]{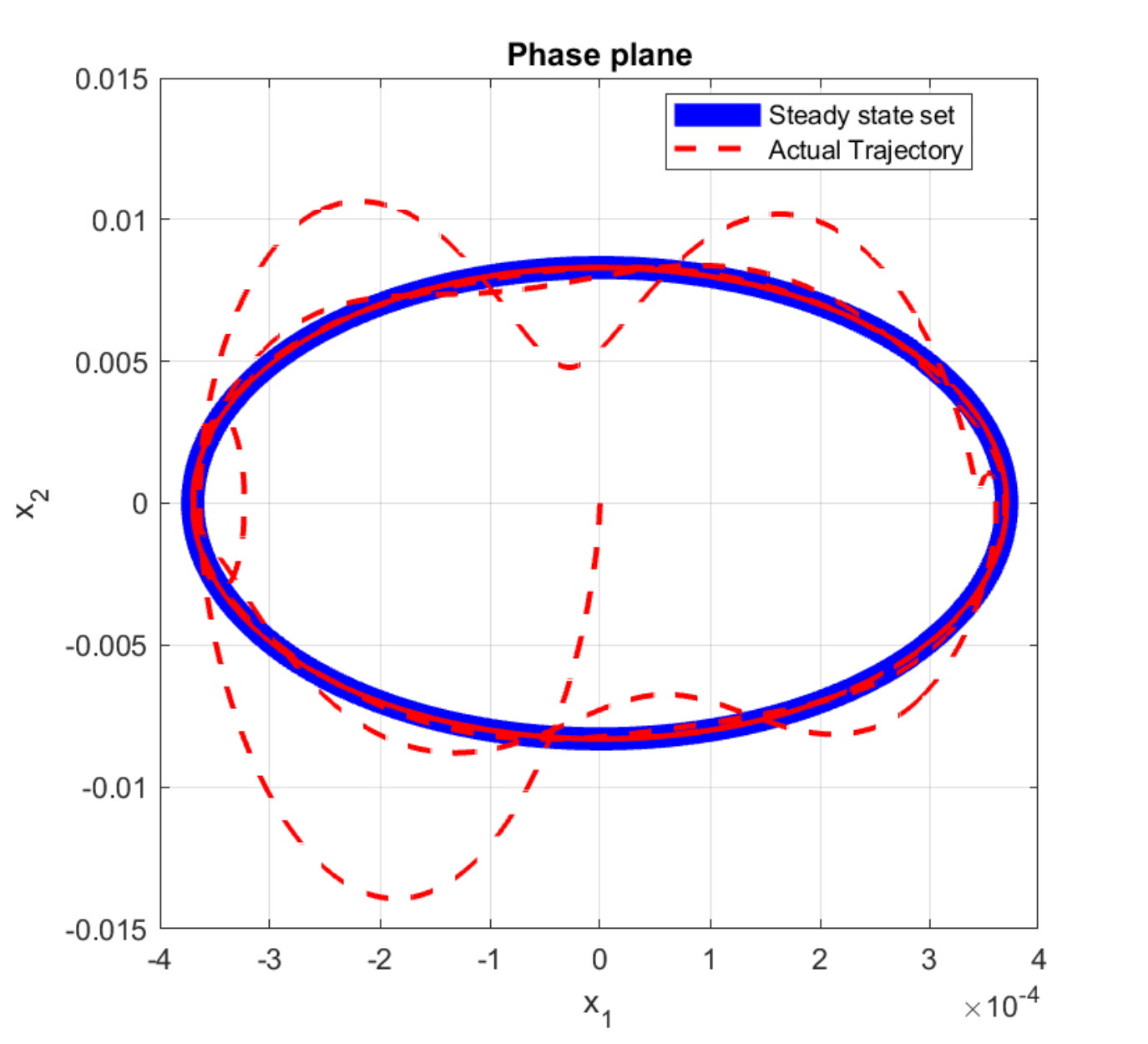}
\caption{The trajectory in the phase plane starting at $[0,0]^T$ eventually converges to the steady-state set.}
\label{fig_fullphaseplot}
\end{figure}

Figure \ref{fig_fullphaseplot} shows the steady-state response of the actual piezoelectric system. This figure gives us an estimate of maximum and minimum displacement. Under the assumption that the unknown nonlinear term is a function of displacement only, it is clear that the set $\Omega \subseteq \mathbb{R}$. For this problem, the set $\Omega$ is the closed interval from minimum displacement to the maximum displacement. For the adaptive estimator, the reproducing kernel implemented in the simulation was selected to be the popular exponential function

\begin{align*}
    \mathcal{K}(x,y) = e^{-\frac{||x-y||^2}{2 \sigma^2}}.
\end{align*}
Thus, $\mathcal{H}_\Omega$ is the set defined as 
\begin{align*}
    \mathcal{H}_\Omega := \overline{\{ \mathcal{K}(x,\cdot) = e^{-\frac{||x-\cdot||^2}{2 \sigma^2}} | x \in \Omega \subseteq \mathbb{R} \}},
\end{align*}
where $\sigma$ is the standard deviation of the radial basis function. For the simulations, we used $\sigma = 1e-9$. A total of $24$ equidistant points were chosen in the interval $\Omega = [-0.00037018, 0.00037026]$ and the kernel functions were centered at these points. It is clear from the state-state trajectory in Figure \ref{fig_fullphaseplot} that the hypotheses for the sufficient condition given in Theorem \ref{thm_sc} are satisfied.

\begin{figure}
     \centering
     \begin{minipage}[c]{0.45\textwidth}
         \centering
         \includegraphics[scale = 0.55]{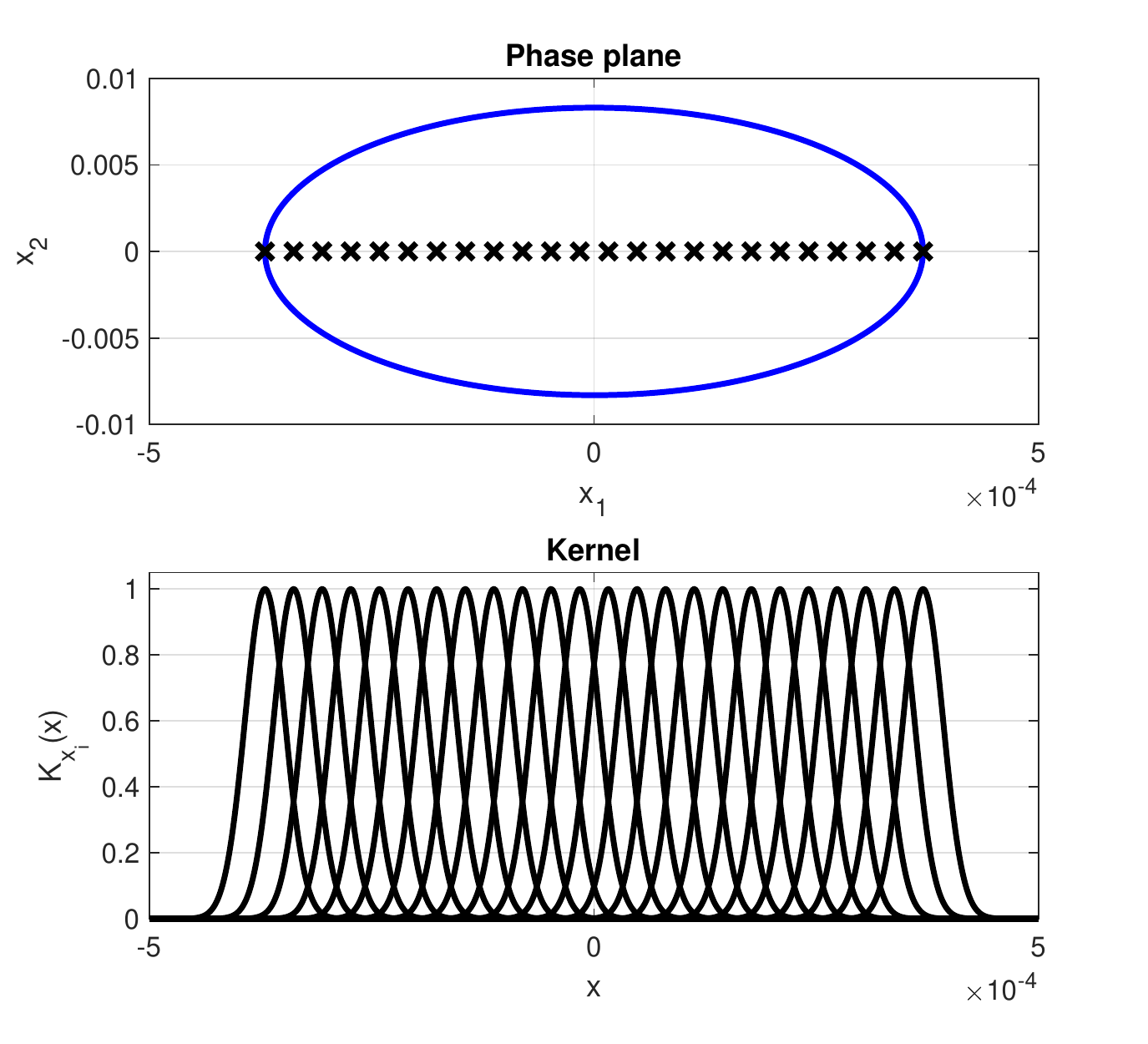}
         \caption{Radial basis functions centered at equidistant points in $\Omega$.}
         \label{fig_kerneldef}
     \end{minipage}
     \begin{minipage}[c]{0.45\textwidth}
         \centering
         \includegraphics[scale = 0.55]{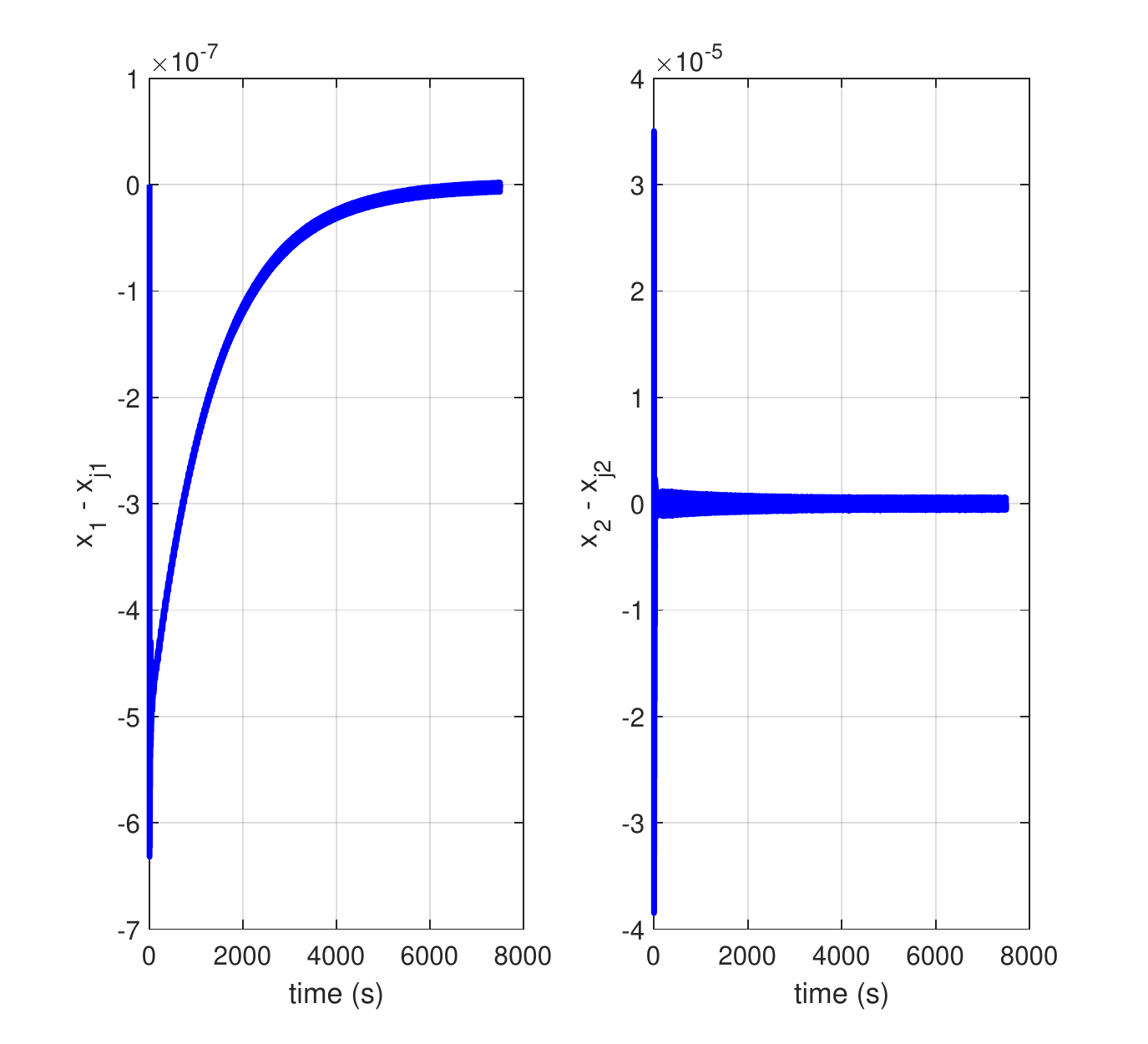}
         \caption{Evolution of state error with time.}
         \label{fig_stateerr}
     \end{minipage}
\end{figure}

\begin{figure}
     \centering
     \begin{subfigure}[c]{0.45\textwidth}
         \centering
         \includegraphics[scale = 0.5]{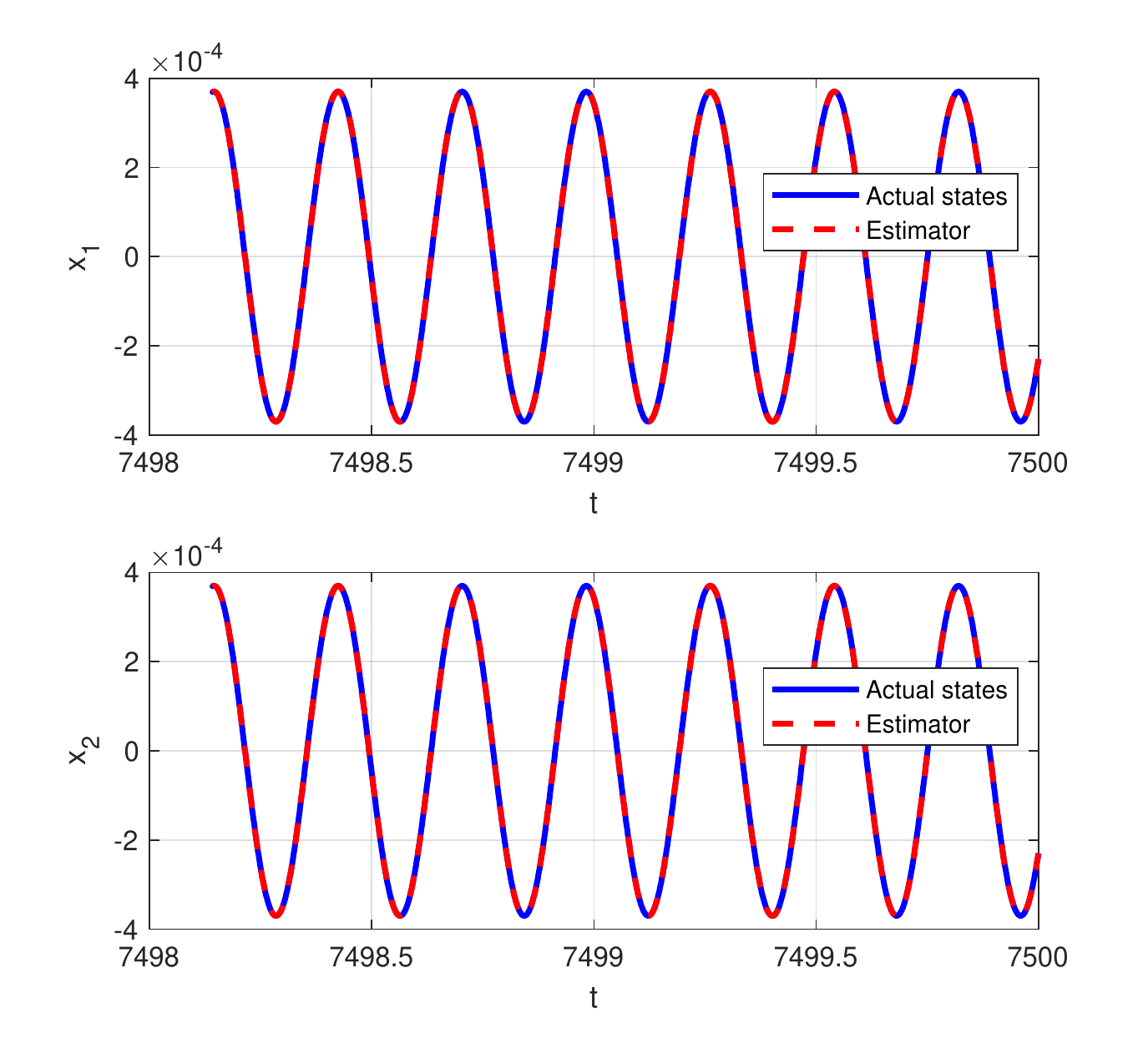}
        \caption{Actual states and state estimate - final 500 timesteps}
         \label{fig_steadystate}
     \end{subfigure}
     \begin{subfigure}[c]{0.45\textwidth}
         \centering
         \includegraphics[scale = 0.5]{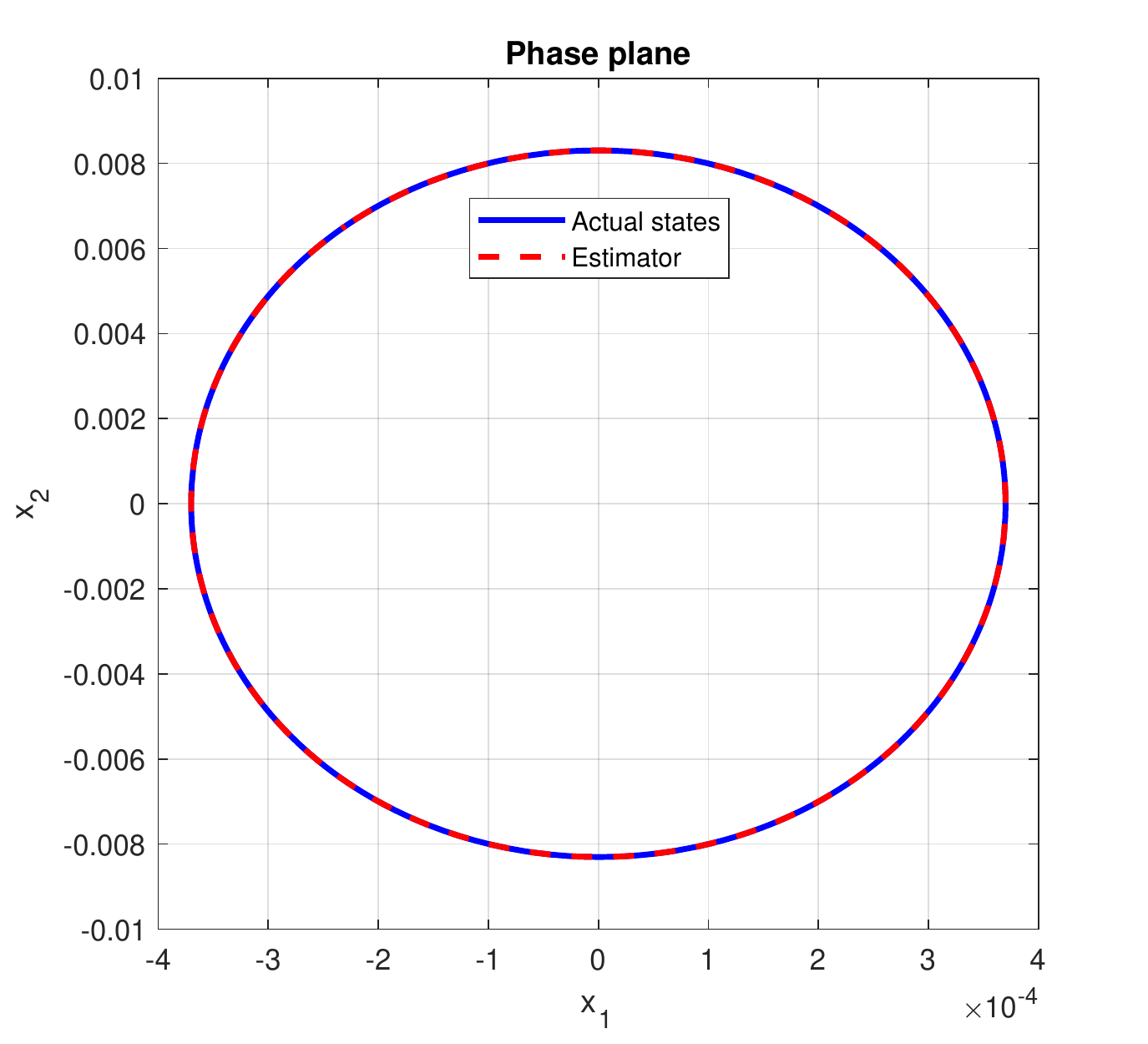}
         \caption{Actual states and state estimates in phase plane after convergence of state error to 0.}
         \label{fig_steadystatepp}
     \end{subfigure}
     \caption{State estimate plots}
\end{figure}

Figure \ref{fig_stateerr} shows the time history of the state errors. As expected, the state errors eventually converge to zero. Figure \ref{fig_steadystate} shows the final 500-time-steps of the actual states and the estimated states. Figure \ref{fig_steadystatepp} shows the corresponding phase plot. It is clear from these plots that the estimator tracks the actual states with almost no error.

\begin{figure}
     \centering
     \begin{subfigure}[c]{0.45\textwidth}
         \centering
         \includegraphics[scale = 0.5]{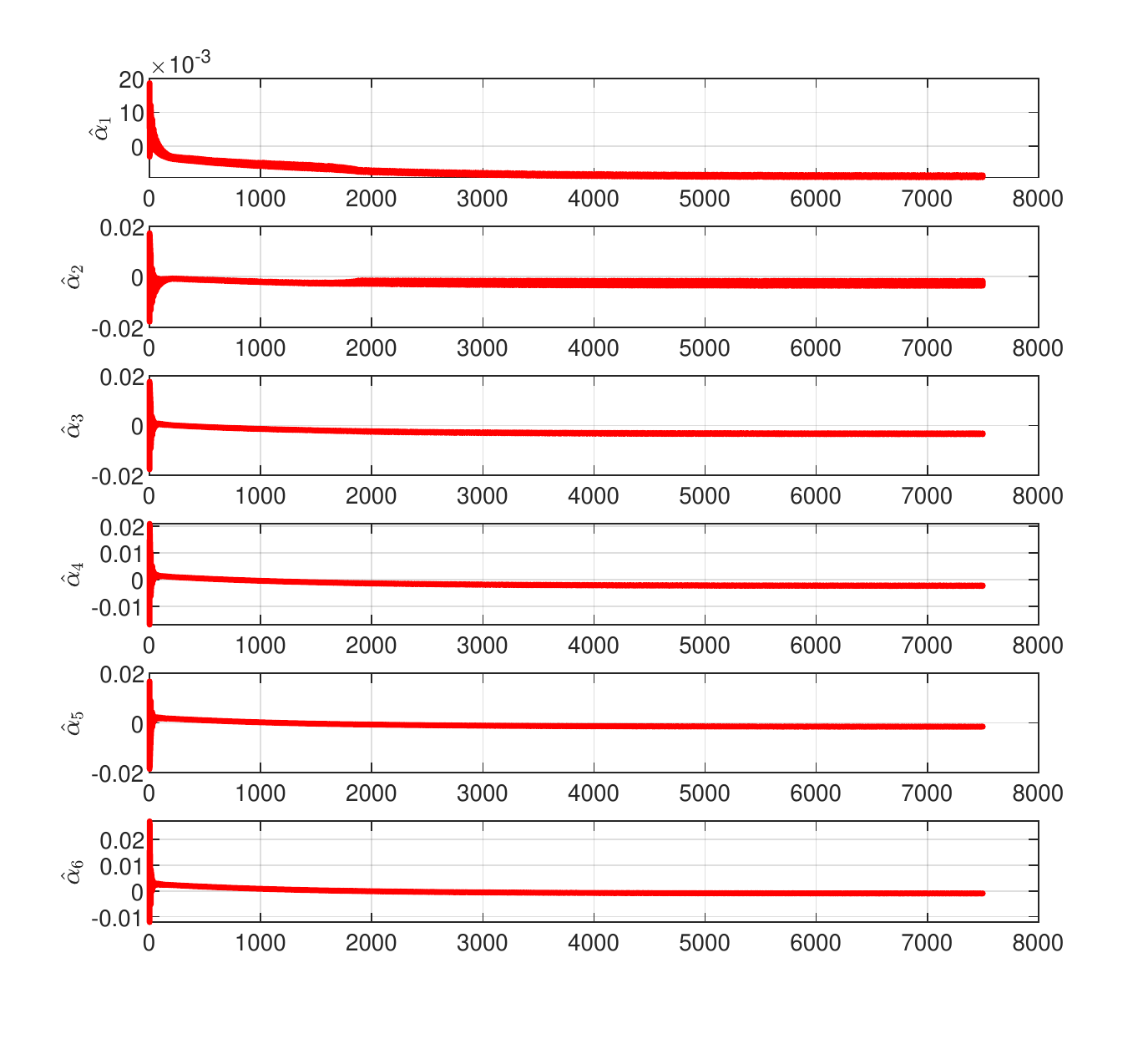}
         \caption{Evolution of $\hat{\alpha}_{1} - \hat{\alpha}_{6}$}
         \label{fig_param1}
     \end{subfigure}
     \begin{subfigure}[c]{0.45\textwidth}
         \centering
         \includegraphics[scale = 0.5]{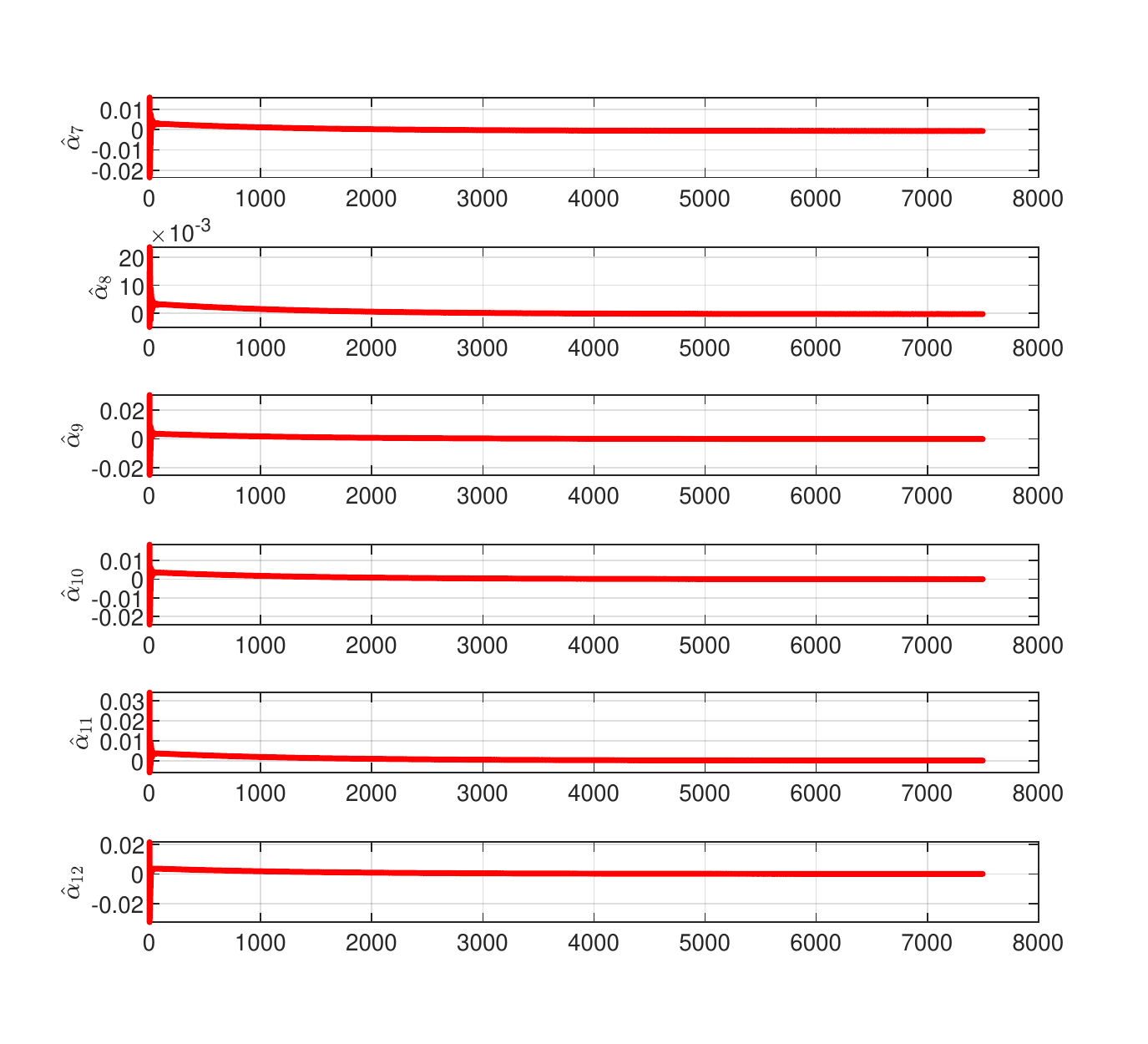}
         \caption{Evolution of $\hat{\alpha}_{7} - \hat{\alpha}_{12}$}
         \label{fig_param2}
     \end{subfigure}
     \caption{Evolution of the parameter estimates $\hat{\alpha}_{1} - \hat{\alpha}_{12}$ with time.}
     \label{fig_param1to12}
\end{figure}

\begin{figure}
     \centering
     \begin{subfigure}[c]{0.45\textwidth}
         \centering
         \includegraphics[scale = 0.5]{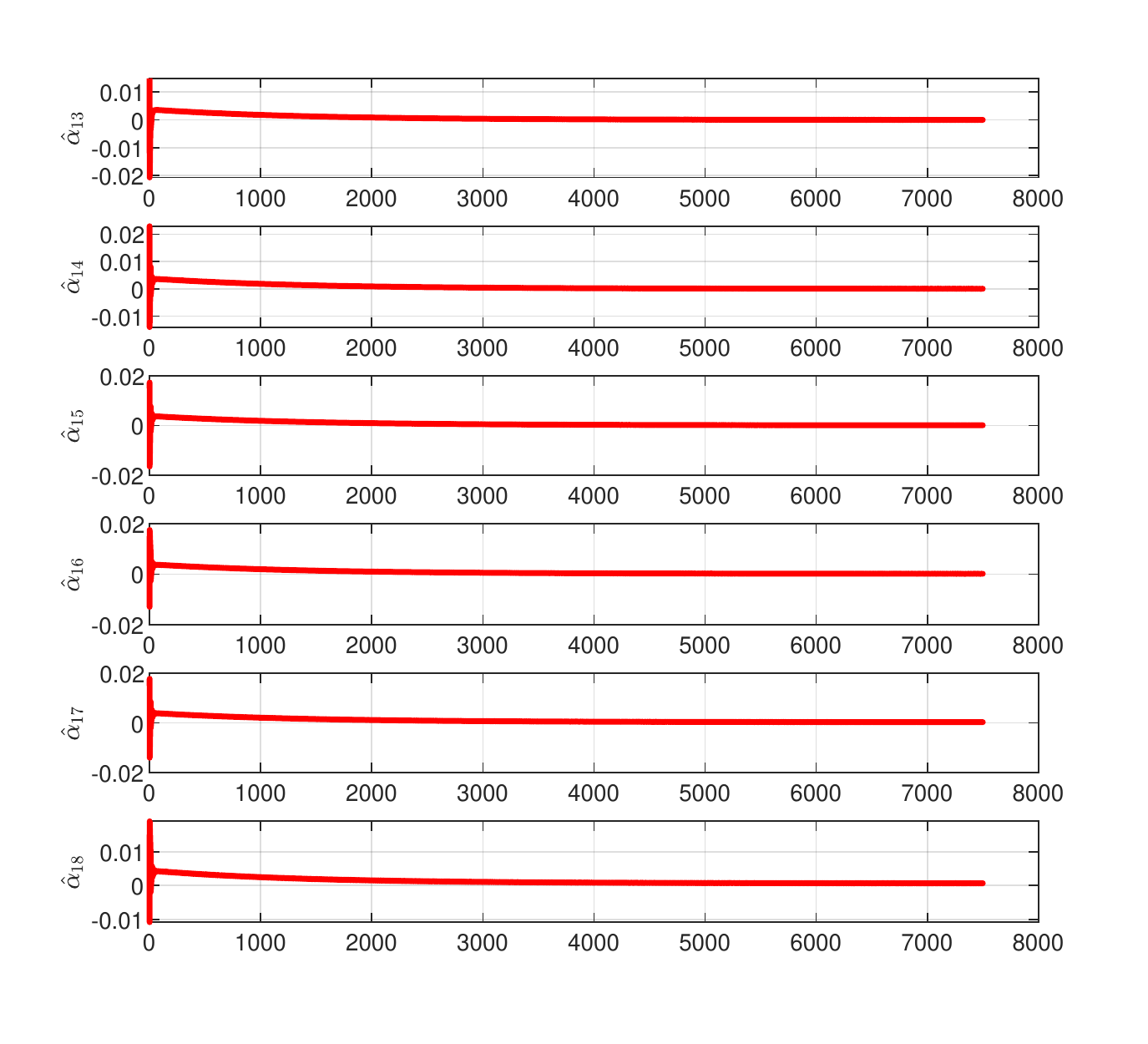}
         \caption{Evolution of $\hat{\alpha}_{13} - \hat{\alpha}_{18}$}
         \label{fig_param3}
     \end{subfigure}
     \begin{subfigure}[c]{0.45\textwidth}
         \centering
         \includegraphics[scale = 0.5]{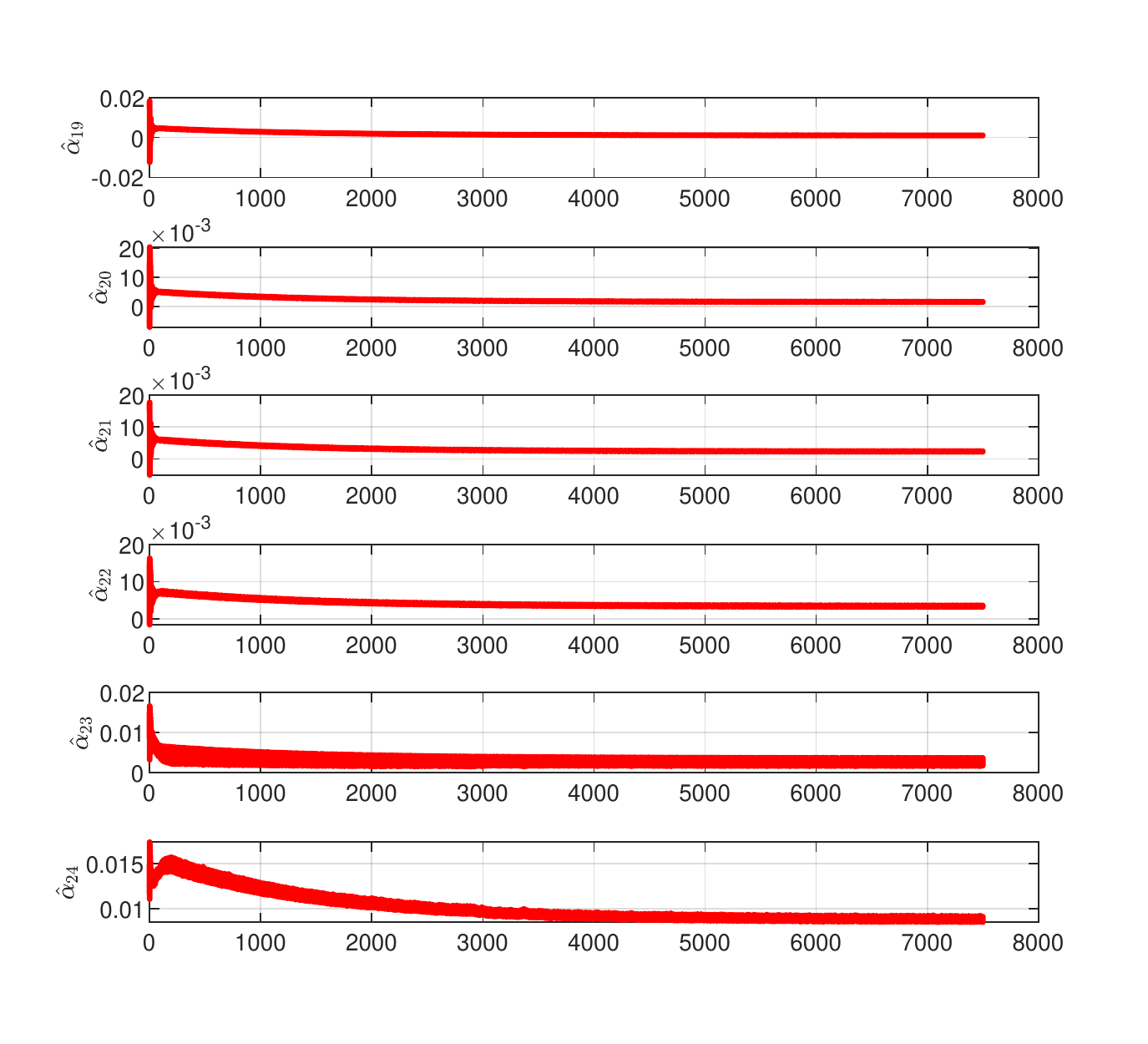}
         \caption{Evolution of $\hat{\alpha}_{19} - \hat{\alpha}_{24}$}
         \label{fig_param4}
     \end{subfigure}
     \caption{Evolution of the parameter estimates $\hat{\alpha}_{13} - \hat{\alpha}_{24}$ with time.}
     \label{fig_param13to24}
\end{figure}

Figures \ref{fig_param1}, \ref{fig_param2}, \ref{fig_param3} and \ref{fig_param4} show the evolution of the parameters. It is clear from the figures that the estimated parameters converge to a constant as time $t \to \infty$.

\begin{figure}
     \centering
     \begin{subfigure}[c]{0.45\textwidth}
         \centering
         \includegraphics[scale = 0.5]{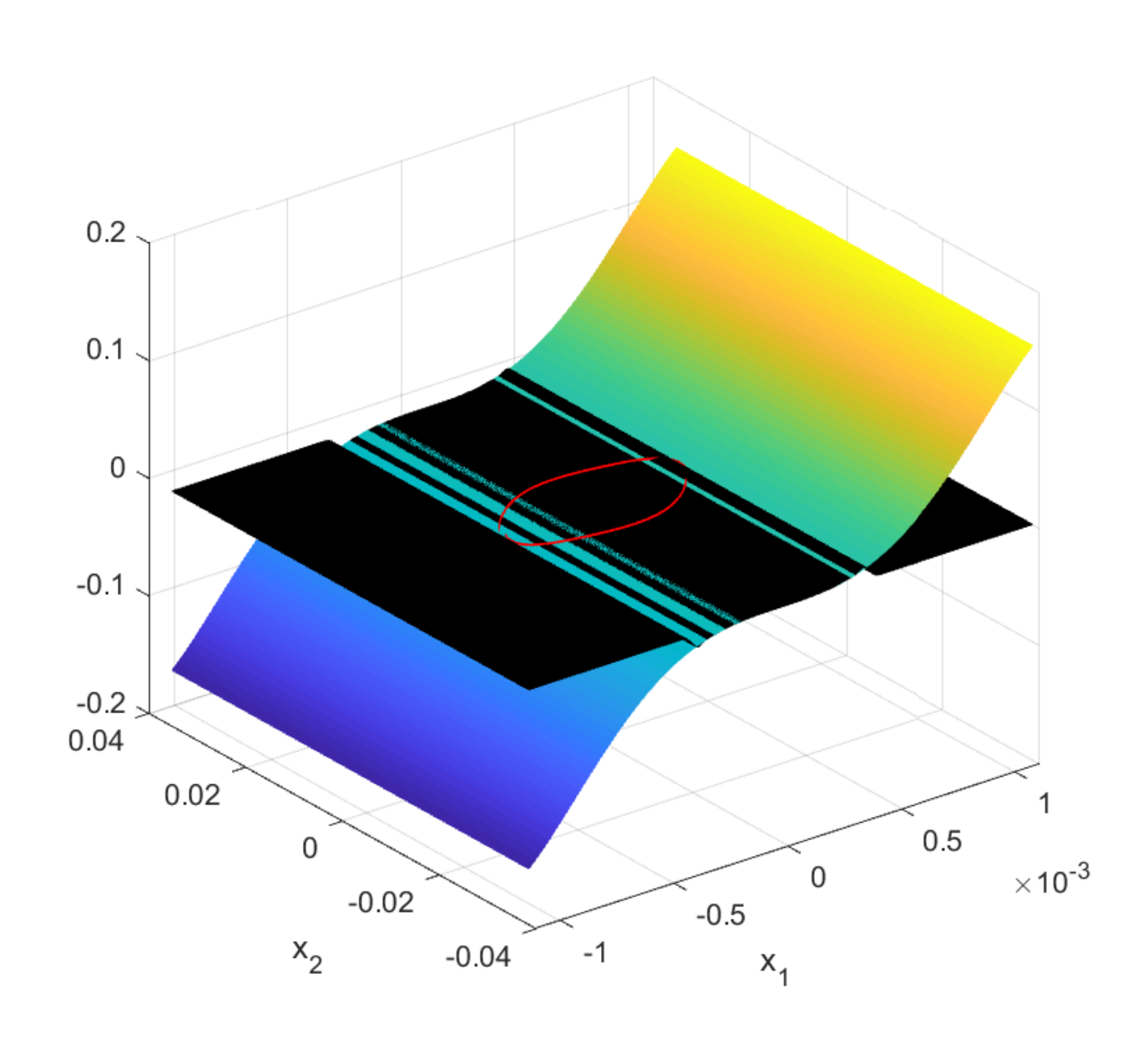}
         \caption{Function estimate and Actual function plot}
         \label{fig_3Dfuncs_tvsest}
     \end{subfigure}
     \begin{subfigure}[c]{0.45\textwidth}
         \centering
         \includegraphics[scale = 0.5]{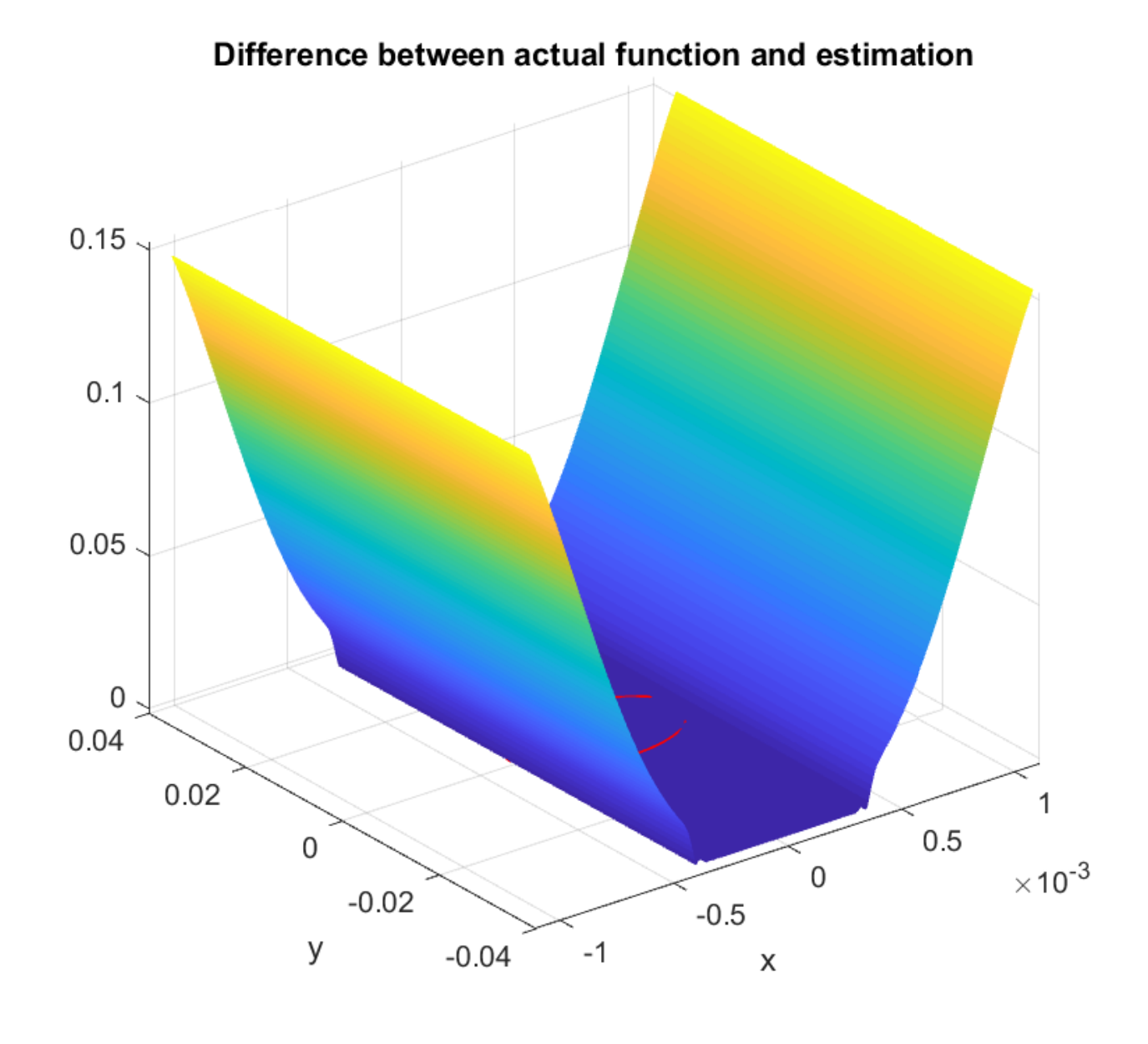}
         \caption{Error between function estimate and actual function}
         \label{fig_3Dfuncs_err}
     \end{subfigure}
     \caption{Function estimate on state-space plots.}
     \label{fig_3Dfuncs}
\end{figure}

The plots of the actual function $f$ and estimated function $\hat{f}_i$ evaluated on the state space $\mathbb{R}^2$ can be seen in Figure \ref{fig_3Dfuncs_tvsest}. The figure shows that the estimated function does not vary along the $x_2$ direction, and this is because of the assumption that the set $\Omega \subseteq \mathbb{R}$. Figure \ref{fig_2Dfuncs_tvsest} shows that the actual and estimated functions agree on $\Omega$. Recall that convergence of the function error is guaranteed in the norm on $\mathcal{H}_{\Omega}$ essentially. This amounts to a guarantee of the pointwise error over the set $\Omega$. No guarantee is made for values outside $\Omega$. See \cite{jia2020a,jia2020b,Kurdila2019PE} for more details on the convergence.

\begin{figure}
\centering
\begin{subfigure}[c]{0.45\textwidth}
\centering
\includegraphics[scale = 0.5]{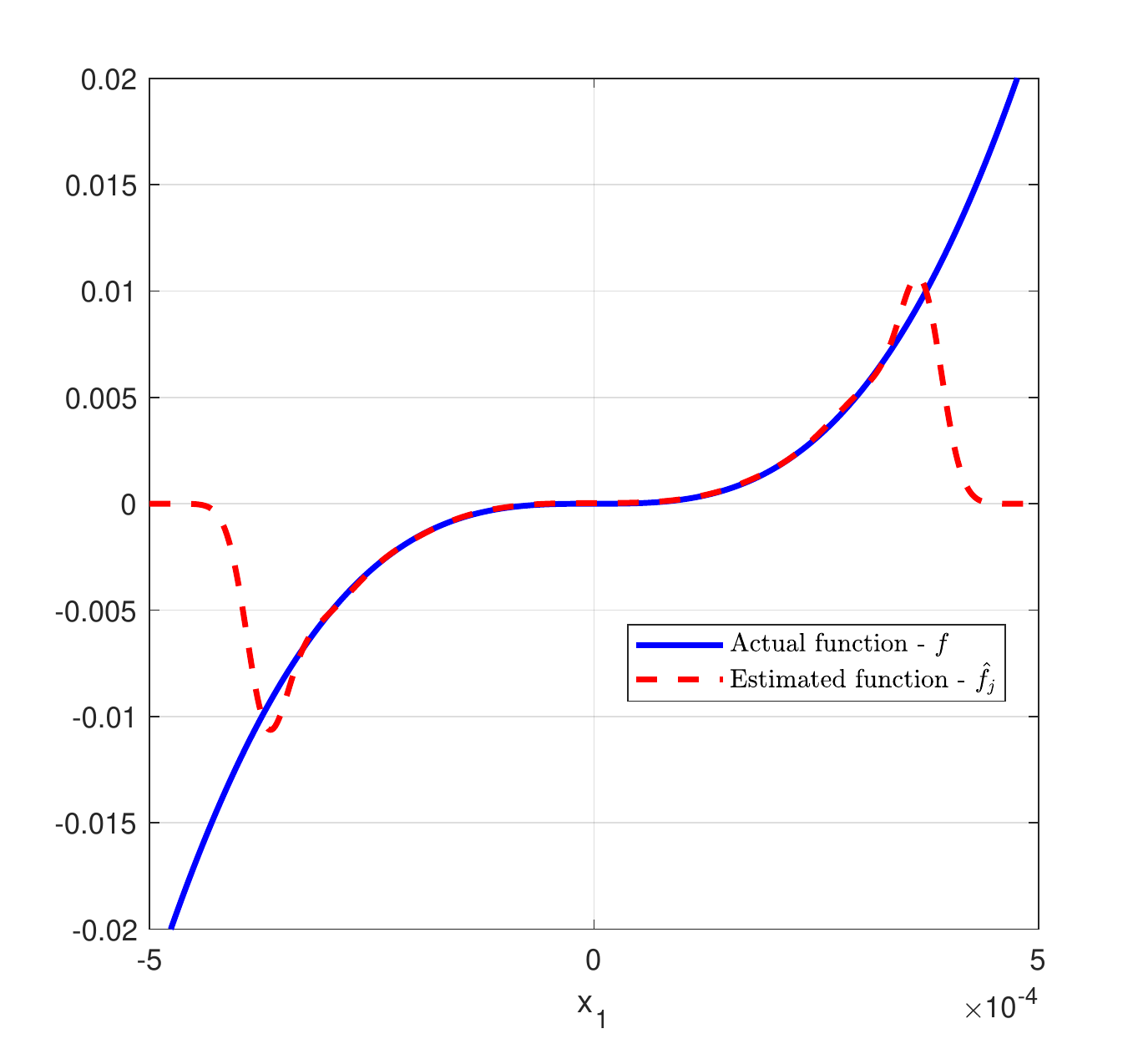}
\caption{Actual function and function estimate.}
\label{fig_2Dfuncs_tvsest}
\end{subfigure}
\begin{subfigure}[c]{0.45\textwidth}
\centering
\includegraphics[scale = 0.5]{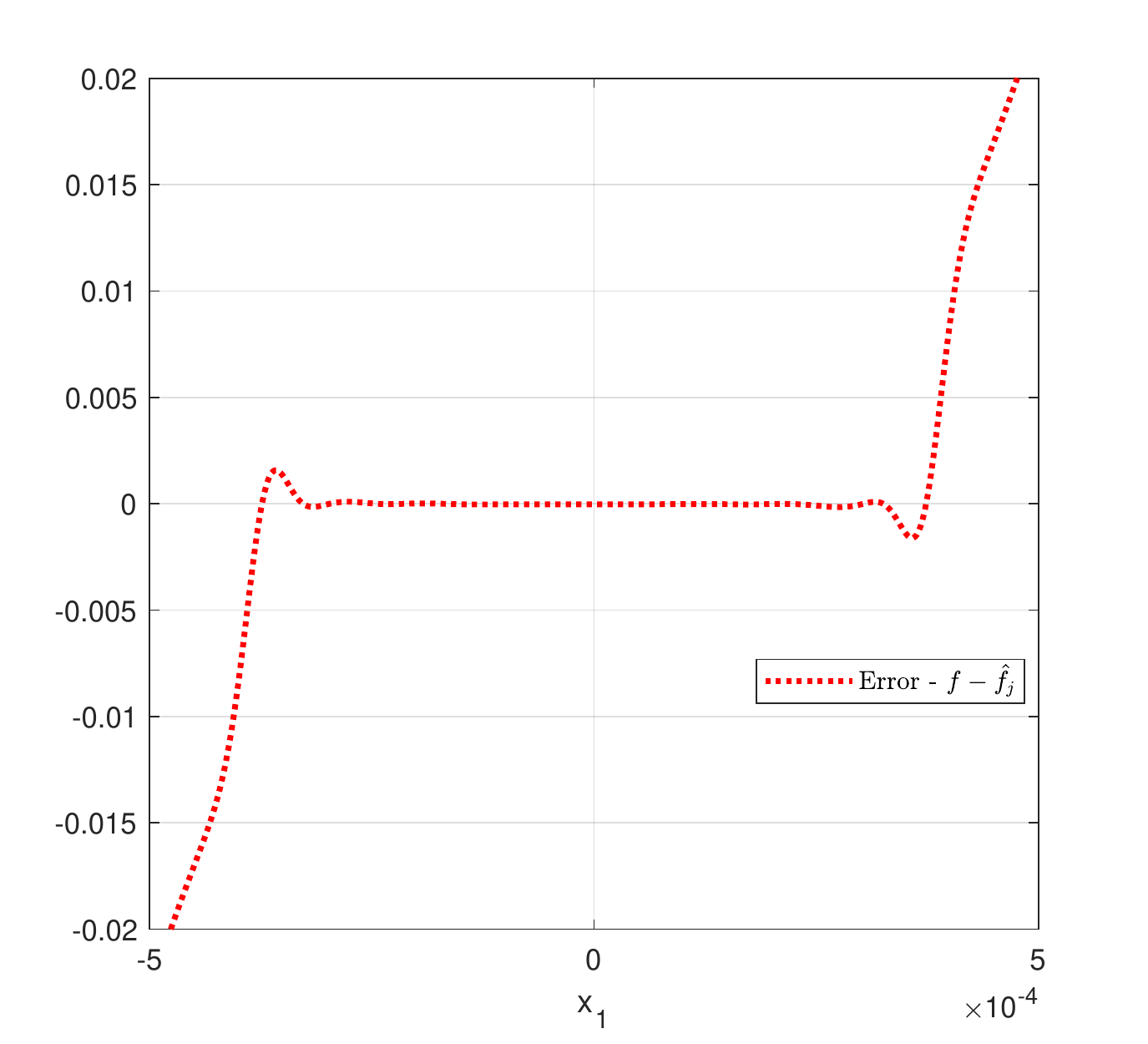}
\caption{Error between actual function and function estimate.}
\label{fig_2Derrfuncs_tvsest}
\end{subfigure}
\caption{Function estimate on $\mathbb{R}$ plots.}
\label{fig_2Dfuncs}
\end{figure}




\section{Conclusion}
This paper has introduced a novel approach to model and estimate uncertain nonlinear piezoelectric oscillators, and the effectiveness of the approach has been validated by testing it on a nonlinear piezoelectric bimorph beam. The nonlinear function used in the numerical study depended only on the displacement, but much of the theory applies to more complex uncertainties. It would be of interest to study the effectiveness of such estimators on more complex oscillators, ones for which unknown nonlinearities depend on all the states. The algorithm discussed in this paper follows a general framework and can be adapted easily to model many other nonlinearities. Robustness of the current algorithm and its effectiveness in the presence of noise would be of great interest and remains to be explored and would complement the findings in the current study.

\appendix
\section{Piezoelectric Oscillator - Governing Equations}
\label{app_EOM}
In this section, we go over the detailed steps involved in the derivation of the infinite-dimensional governing equation of the piezoelectric oscillator shown in Figure \ref{fig_pbeam}. The kinetic energy and the electric potential are given by Equation \ref{eq_kinEn} and Equation \ref{eq_elecPot}, respectively. Using Hamilton's principle, we get the variational identity

\begin{align}
    \delta \int_{t_0}^{t_1} (T - \mathcal{V}_\mathcal{H}) dt 
    &= 
    \delta \int_{t_0}^{t_1} \left \{ 
    \left[ 
    \frac{1}{2} \mathscr{m} \int_0^l (\dot{w} + \dot{\mathscr{z}})^2 dx
    \right]
    \right.
    \notag
    \\
    & \qquad -
    \left[ 
    \frac{1}{2} C_b I_b \int_0^l (w'')^2 dx 
    + 2 a_{(0,2)} \int_{a}^{b} (w'')^2 dx + 2 a_{(2,4)} \int_{a}^{b} (w'')^4 dx \right.
    \notag
    \\
    & \qquad \qquad \qquad \left. \left.
    + 2 b_{(1,1)} \left[ \int_{a}^{b} w'' dx \right] E_z + 2 b_{(3,1)} \left[ \int_{a}^{b} (w'')^3 dx \right] E_z - 2 b_{(0,2)} E_z^2
    \right]
    \right \} dt = 0.
    \label{eq_var_stat_ham_prin}
\end{align}

The above variational statement can be rewritten as

\begin{align}
    \delta \int_{t_0}^{t_1} L dt &= \int_{t_0}^{t_1} \left \{
    \int_0^l (m \dot{w} \delta \dot{w} + m \dot{\mathscr{z}} \delta \dot{w}) dx 
    - \int_0^l C_b I_b w'' \delta w'' dx \right.
    - 4a_{(0,2)} \int_{a}^{b} w'' \delta w'' dx 
    \notag
    \\
    & \hspace{1.5cm} - 8 a_{(2,4)} \int_{a}^{b} (w'')^3 \delta w'' dx 
    - 2b_{(1,1)} \left( \int_{a}^{b} (\delta w'') dx \right) E_z
    - 2b_{(1,1)} \left( \int_{a}^{b} w'' dx \right) \delta E_z
    \notag
    \\
    & \hspace{1.5cm} \left.
    - 6b_{(3,1)} \left( \int_{a}^{b} (w'')^2 \delta w'' dx \right) E_z
    - 2b_{(3,1)} \left( \int_{a}^{b} (w'')^3 dx \right) \delta E_z
    + 4b_{(0,2)} E_z \delta E_z
    \right \} dt = 0
    \label{eq_vs}
\end{align}

After integrating the above statement by parts, we get

\begin{align*}
    &\int_{t_0}^{t_1} \Bigg \{
    - \int_0^l m \ddot{w} \delta w dx 
    - \int_0^l m \ddot{\mathscr{z}} \delta w dx
    \\
    & \hspace{2cm} - \left. C_b I_b w'' \delta w' \right|_{0}^{l}
    + \left. C_b I_b w''' \delta w \right|_{0}^{l}
    - \int_0^l C_b I_b w'''' \delta w dx
    \\
    & \hspace{2cm} - \left. 4a_{(0,2)} \chi_{[a,b]} w'' \delta w' \right|_{0}^{l}
    + \left. 4a_{(0,2)} \chi_{[a,b]} w''' \delta w \right|_{0}^{l}
    - 4 a_{(0,2)} \int_{0}^{l} \chi_{[a,b]} w'''' \delta w dx
    \\
    & \hspace{2cm} - \left. 8a_{(2,4)} \chi_{[a,b]} (w'')^3 \delta w' \right|_{0}^{l}
    + \left. 8a_{(2,4)} (\chi_{[a,b]} (w'')^3)' \delta w \right|_{0}^{l}
    - 8a_{(2,4)} \int_{0}^{l} (\chi_{[a,b]} (w'')^3)'' \delta w dx
    \\
    & \hspace{2cm} - 2 b_{(1,1)} \left( \left. \chi_{[a,b]} \delta w' \right|_0^l \right) E_z + 2 b_{(1,1)} \left( \left. \chi_{[a,b]}' \delta w \right|_0^l \right) E_z
    - 2 b_{(1,1)} \left( \int_0^l \chi_{[a,b]}'' \delta w dx \right) E_z
    \\
    & \hspace{2cm} - 2 b_{(1,1)} \left( \int_0^l \chi_{[a,b]} w'' dx \right) \delta E_z
    \\
    & \hspace{2cm} \left. - 6 b_{(3,1)} \chi_{[a,b]} (w'')^2 \delta w' \right|_{0}^{l} E_z
    + \left. 6 b_{(3,1)} (\chi_{[a,b]} (w'')^2)' \delta w \right|_{0}^{l} E_z
    - 6 b_{(3,1)} \left( \int_{0}^{l} (\chi_{[a,b]} (w'')^2)'' \delta w dx \right) E_z
    \\
    & \hspace{2cm}
    - 2 b_{(3,1)} \left( \int_{a}^{b} (w'')^3 dx \right) \delta E_z
    \\
    & \hspace{2cm}
    + 4 b_{(0,2)} E_z \delta E_z
    \Bigg \} dt = 0.
\end{align*}

Note, in the above statement, the term $\chi_{[a,b]}$ is called the characteristic function of $[a,b]$ and is defined as
\begin{align}
    \chi_{[a,b]}(x) :=
    \left \{
    \begin{array}{cc}
        1 & \text{if } x \in [a,b], \\
        0 & \text{if } x \notin [a,b].
    \end{array}
    \right.
    \label{eq_chaeq}
\end{align}

Rearranging the terms in the above variational statement results in the expression

\begin{align*}
    &\int_{t_0}^{t_1} \Bigg \{
    - \int_0^l \bigg[ m \ddot{w} + m \ddot{\mathscr{z}} 
    + C_b I_b w'''' 
    + 4 a_{(0,2)} \left( \chi_{[a,b]} w'' \right)'' 
    + 8 a_{(2,4)} (\chi_{[a,b]} (w'')^3)'' 
    \\
    & \hspace{4.5cm} 
    + 2 b_{(1,1)} \chi_{[a,b]}'' E_z 
    + 6 b_{(3,1)} (\chi_{[a,b]} (w'')^2 )'' E_z \bigg] \delta w dx
    \\
    & \hspace{1.5cm} 
    - \left[ 2b_{(1,1)} \left( \int_0^l \chi_{[a,b]} w'' dx \right)
    + 2b_{(3,1)} \left( \int_{0}^{l} \chi_{[a,b]} (w'')^3 dx \right)
    - 4b_{(0,2)} E_z
    \right] \delta E_z
    \\
    & \hspace{1.5cm} 
    - \left \{
    C_b I_b w'' + 4a_{(0,2)} \chi_{[a,b]} w'' + 8a_{(2,4)} \chi_{[a,b]} (w'')^3 + 2 b_{(1,1)} \chi_{[a,b]} E_z + 6 b_{(3,1)} \chi_{[a,b]} (w'')^2
    \right \} \delta w' \big|_0^l
    \\
    & \hspace{1.5cm} 
    + \left \{
    C_b I_b w''' + 4a_{(0,2)} \left(\chi_{[a,b]} w''\right)' + 8a_{(2,4)} \left( \chi_{[a,b]} (w'')^3 \right)' \right.
    \\
    & \hspace{4.5cm} 
    \left.
    + 2 b_{(1,1)} \chi_{[a,b]}' E_z + 6 b_{(3,1)} \left( \chi_{[a,b]} (w'')^2 \right)'
    \right \} \delta w \bigg|_0^l
    \Bigg \} dt = 0.
\end{align*}

Since the variation of $w$ and $E_z$ are arbitrary, we can conclude that the equations of motion of the nonlinear piezoelectric cantilevered bimorph have the form shown in Equations \ref{eq_eom1} and \ref{eq_eom2}.

\section{Single Mode Approximation of the Piezoelectric Oscillator Equations}
\label{app_SMEOM}
As mentioned earlier, the effects of nonlinearity in piezoelectric oscillators are most noticeable near the natural frequency of the system. Hence, single-mode models are sufficient to model the dynamics as long as the range of input excitation is restricted to a band around the first natural frequency. Let us introduce the single-mode approximation $w(x,t) = \psi(x) u(t)$. To make calculations easier, let us introduce this approximation into the variational statement shown in Equation \ref{eq_vs}. Further, note that

\begin{align*}
    \int_{t_0}^{t_1} \int_0^l  m \dot{w} \delta \dot{w} dx dt = - \int_{t_0}^{t_1} \int_0^l m \ddot{w} \delta w dx dt,  
    & & 
    \int_{t_0}^{t_1} \int_0^l  m \dot{\mathscr{z}} \delta \dot{w} dx dt = - \int_{t_0}^{t_1} \int_0^l m \ddot{\mathscr{z}} \delta w dx dt.
\end{align*}

After introducing the approximation for $w(x,t)$ into the variational statement in Equation \ref{eq_vs} and using the equations shown above, we get the variational statement

\begin{align*}
    & 0 = \delta \int_{t_0}^{t_1} \Bigg \{
    -m \left( \int_0^l \psi^2(x) dx \right) \ddot{u} \delta u
    -m \left( \int_0^l \psi(x) dx \right) \ddot{\mathscr{z}} \delta u
    \\
    & \hspace{2cm}
    -C_b I_b \left( \int_0^l \left( \psi''(x) \right)^2 dx \right) u \delta u
    -4a_{(0,2)} \left( \int_0^l \chi_{[a,b]} \left( \psi''(x) \right)^2 \right) u \delta u
    \\
    & \hspace{2cm}
    -8a_{(2,4)} \left( \int_0^l \chi_{[a,b]} \left( \psi''(x) \right)^4 dx \right) u^3 \delta u
    -2b_{(1,1)} E_z \left( \int_0^l \chi_{[a,b]} \psi''(x) dx \right) \delta u
    \\
    & \hspace{2cm}
    -2b_{(1,1)} \left( \int_0^l \chi_{[a,b]} \psi''(x) dx \right) u \delta E_z
    -6b_{(3,1)} \left( \int_0^l \chi_{[a,b]} \left( \psi''(x) \right)^3 \right) u^2 E_z \delta u
    \\
    & \hspace{2cm}
    -2b_{(3,1)} \left( \int_0^l \chi_{[a,b]} \left( \psi''(x) \right)^3 dx \right) u^3 \delta E_z
    + 4 b_{(0,2)} E_z \delta E_z
    \Bigg \} dt.
\end{align*}

Rearranging the terms in the above variational statement, we get

\begin{align*}
    0 = \int_{t_0}^{t_1} \Bigg \{ & - \bigg[
    m \left( \int_0^l \psi^2(x) dx \right) \ddot{u}
    + m \left( \int_0^l \psi(x) dx \right) \ddot{\mathscr{z}}
    + C_b I_b \left( \int_0^l \left( \psi''(x) \right)^2 dx \right) u 
    \\
    & \hspace{3cm}
    + 4a_{(0,2)} \left( \int_0^l \chi_{[a,b]} \left( \psi''(x) \right)^2 \right) u 
    + 8a_{(2,4)} \left( \int_0^l \chi_{[a,b]} \left( \psi''(x) \right)^4 dx \right) u^3 
    \\
    & \hspace{3cm}
    + 2b_{(1,1)} E_z \left( \int_0^l \chi_{[a,b]} \psi''(x) dx \right) 
    + 6b_{(3,1)} \left( \int_0^l \chi_{[a,b]} \left( \psi''(x) \right)^3 \right) u^2 E_z
    \bigg] \delta u
    \\
    & - \bigg[
    2b_{(1,1)} \left( \int_0^l \chi_{[a,b]} \psi''(x) dx \right) u
    + 2b_{(3,1)} \left( \int_0^l \chi_{[a,b]} \left( \psi''(x) \right)^3 \right) u^3
    - 4b_{(0,2)} E_z
    \bigg] \delta E_z
    \Bigg \} dt
\end{align*}

Thus, the approximated equation of motion are

\begin{align*}
    & \underbrace{m \left( \int_0^l \psi^2(x) dx \right)}_{M} \ddot{u}
    + \underbrace{m \left( \int_0^l \psi(x) dx \right)}_{P} \ddot{\mathscr{z}}
    + \underbrace{C_b I_b \left( \int_0^l \left( \psi''(x) \right)^2 dx \right)}_{K_b} u 
    \\
    & \hspace{2cm}
    + \underbrace{4a_{(0,2)} \left( \int_0^l \chi_{[a,b]} \left( \psi''(x) \right)^2 \right)}_{K_p} u 
    + \underbrace{8a_{(2,4)} \left( \int_0^l \chi_{[a,b]} \left( \psi''(x) \right)^4 dx \right)}_{K_N} u^3 
    \\
    & \hspace{2cm}
    + \underbrace{2b_{(1,1)} \left( \int_0^l \chi_{[a,b]} \psi''(x) dx \right)}_{B} E_z 
    + \underbrace{6b_{(3,1)} \left( \int_0^l \chi_{[a,b]} \left( \psi''(x) \right)^3 \right)}_{Q_{N}} u^2 E_z = 0,
    \\
    & \underbrace{2b_{(1,1)} \left( \int_0^l \chi_{[a,b]} \psi''(x) dx \right)}_{B} u(t)
    + \underbrace{2b_{(3,1)} \left( \int_0^l \chi_{[a,b]} \left( \psi''(x) \right)^3 \right)}_{B_N} u^3(t)
    = \underbrace{4b_{(0,2)}}_{\mathcal{C}} E_z.
\end{align*}
These calculations generate the approximated Equations \ref{eq_SMEOM1} and \ref{eq_SMEOM2}.

\bibliographystyle{unsrt}  
\bibliography{RKHS}

\end{document}